\documentclass[final]{siamart0516}

\usepackage{algpseudocode}
\algrenewcommand\algorithmicrequire{\textbf{Input:}}
\algrenewcommand\algorithmicensure{\textbf{Output:}}
\usepackage{xpatch}
\makeatletter
\xpatchcmd{\algorithmic}{\ALG@tlm\z@}{\ALG@tlm\z@\leftmargin 0.5cm}{}{}
\renewcommand{\fnum@algorithm}{\fname@algorithm}
\makeatother

\usepackage{subfig}
\usepackage{enumitem}
\usepackage{comment}

\newsiamthm{remark}{Remark}
\newsiamthm{example}{Example}

\usepackage{tikz} 
\usetikzlibrary{arrows,positioning,fit,calc,shapes,spy,patterns}
\usepackage{pgfplots}
\usepackage{framed}
\usepackage{esint}

\usepackage{amssymb}
\usepackage{mathabx}

\usepackage{commath} 

\usepackage{calc}
\usepackage{accents}
\newcommand{\R}{\ensuremath{\mathbb R}}

\newcommand{\Mbb}{\mathbb{M}}
\newcommand{\Tbb}{\mathbb{T}}
\newcommand{\Sbb}{\mathbb{S}}
\newcommand{\bx}{\mathbf{x}}

\DeclareMathOperator{\tr}{tr}
\DeclareMathOperator{\sym}{{\rm sym}}

\DeclareMathOperator{\ddiv}{div}
\DeclareMathOperator{\ccurl}{curl}
\DeclareMathOperator{\ddev}{dev}
\DeclareMathOperator{\rrot}{rot}

\newcommand{\Tcal}{\ensuremath{\mathcal T}}

\newcommand{\Ecal}{\ensuremath{\mathcal E}}
\newcommand{\Acal}{\ensuremath{\mathcal A}}
\newcommand{\Fcal}{\ensuremath{\mathcal F}}
\newcommand{\Vcal}{\ensuremath{\mathcal V}}

\newcommand*{\Scale}[2][4]{\scalebox{#1}{$#2$}}%

\newcommand{\vfieldv}{\boldsymbol{v}}
\newcommand{\vfieldu}{\boldsymbol{u}}
\newcommand{\vfieldw}{\boldsymbol{w}}
\newcommand{\mfieldt}{\boldsymbol{\tau}}

\newcommand{\mfieldA}{\boldsymbol{A}}
\newcommand{\mfieldE}{\boldsymbol{E}}
\newcommand{\mfieldB}{\boldsymbol{B}}
\newcommand{\mfieldI}{\boldsymbol{I}}
\newcommand{\mfieldxi}{\boldsymbol{\xi}}
\newcommand{\mfieldzeta}{\boldsymbol{\zeta}}
\newcommand{\mfieldchi}{\boldsymbol{\chi}}
\newcommand{\nvec}{\boldsymbol{n}}
\newcommand{\tvec}{\boldsymbol{t}}

\newcommand{\Vspace}{\boldsymbol{V}}
\newcommand{\Vspaceh}{\Vspace_h}
 \newcommand{\Pif}{\Pi_f}
  \newcommand{\Pifn}{\Pi_n}
 \newcommand{\Pifsym}{\Pi_{f,{\rm sym}}}

\def\XXint#1#2#3{{\setbox0=\hbox{$#1{#2#3}{\int}$}
\vcenter{\hbox{$#2#3$}}\kern-.5\wd0}}

\newcommand{\Lnorm}[1]{\big\lVert #1 \big\rVert}
\newlength{\raisebulletlen}
\setbox1=\hbox{$\bullet$}\setbox2=\hbox{\tiny$\bullet$}
\newcommand{\T}{\mathcal{T}}

\title{Conforming finite element  DIVDIV complexes and the application for the linearized Einstein-Bianchi system\thanks{\funding{The first author was supported by the NSFC Projects 11625101 and 11421101. The   third author was  supported  by the project {\em Approximation and reconstruction of stresses in the deformed configuration for hyperelastic material models} (STA 402/14-1) by the DFG via the priority program 1748 {\em Reliable Simulation Techniques in Solid Mechanics, 
Development of Non-standard Discretization Methods, Mechanical
and Mathematical Analysis}.}}}
 \author{Jun Hu\thanks{LMAM and School of Mathematical Sciences, Peking University,  100871 Beijing,  P. R. China (\email{hujun@math.pku.edu.cn}, \email{lyz2015@pku.edu.cn}).} \and
 YiZhou  Liang\footnotemark[2] \and Rui Ma\thanks{Fakult\"{a}t  f\"ur Mathematik, Universit\"at Duisburg-Essen, 45127 Essen, Germany (\email{ rui.ma@uni-due.de})}}

\begin{document}
\maketitle

\begin{abstract}
This paper presents the first family of conforming finite element  $\ddiv\ddiv$ complexes on tetrahedral grids in three dimensions. In these complexes, finite element spaces of $H(\ddiv\ddiv,\Omega;\Sbb)$ are from a current preprint [Chen and Huang, arXiv: 2007.12399, 2020] while finite element spaces of both $H(\sym\ccurl,\Omega;\Tbb)$ and $H^1(\Omega;\R^3)$ are newly constructed here. It is proved that these finite element complexes are exact. As a result, they can be used to discretize the linearized Einstein-Bianchi system within the dual formulation.
\end{abstract}

\section{Introduction}
The linearized Einstein-Bianchi system from \cite{Quenne} reads
\begin{align*}
\dot{\mfieldE}+\ccurl\mfieldB=0,\ddiv\mfieldE=0,\\
\dot{\mfieldB}-\ccurl\mfieldE=0,\ddiv\mfieldB=0
\end{align*}  
with symmetric and traceless tensor fields $\mfieldE$ and $\mfieldB$, respectively. By introducing a new variable $\sigma(t)=\int_0^t\ddiv\ddiv\mfieldE\,ds$, the linearized Einstein-Bianchi system can be realized as a Hodge wave equation
\begin{align}\label{eq:strongformEB}
\begin{cases}
\dot{\sigma}=\ddiv\ddiv\mfieldE,\\
\dot{\mfieldE}=-\nabla\nabla\sigma-\sym\ccurl\mfieldB,\\
\dot{\mfieldB}=\ccurl\mfieldE.
\end{cases}
\end{align}
Given initial conditions $\sigma(0),\mfieldE(0)$ and $\mfieldB(0)$, and with appropriate boundary conditions,   \eqref{eq:strongformEB} is well posed (see \cite{Quenne}).
The weak formulation of \eqref{eq:strongformEB} in  \cite{Quenne}  introduces some Lagrange  multiplies to reduce the constraint of the symmetry of the electric field  $\mfieldE$ and high derivatives of $\sigma$. In \cite{huliang2020}, the first family of conforming finite element spaces of $H(\ccurl,\Omega;\Sbb)$ on a bounded polyhedral domain  $\Omega\subset\R^3$ is constructed as well as those finite element spaces associated with the  following Gradgrad complex from \cite{arnold2021complexes,PaulyZ2020}
\begin{align}\label{eq:gradgradcomplex}\Scale[0.9]{
P_1(\Omega)\xrightarrow[]{ \subset}H^2(\Omega) \xrightarrow[]{\nabla\nabla}H( \ccurl,\Omega;\Sbb) \xrightarrow[]{\ccurl}H(\ddiv, \Omega;\Tbb) \xrightarrow[]{ \ddiv }L^2(\Omega;\R^3) \rightarrow 0,}
\end{align}  
where the space $H(\ccurl,\Omega;\Sbb)$ consists of square-integrable tensors with square-integrable curl, taking values in the space of $\Sbb$ of symmetric matrices, and the space $H(\ddiv,\Omega;\Tbb)$ consists of square-integrable tensors with square-integrable divergence, taking value in the space $\Tbb$ of traceless matrices.
The finite element Gradgrad complexes  are utilized in \cite{huliang2020} to discretize \eqref{eq:strongformEB} with a   weak formulation such that   $\mfieldE$ is sought in $C^0([0,T],H(\ccurl,\Omega;\Sbb))$ with the strongly symmetric constraint, $\sigma$ is  in $C^0([0,T],H^2(\Omega))$ and the magnetic tensor field $\mfieldB$ is in $C^1([0,T],L^2(\Omega;\Tbb))$. The polynomial degree of the lowest order  case for the pair $(\sigma,\mfieldE,\mfieldB)$ is $P_9(K)-P_7(K;\Sbb)-P_6(K;\Tbb)$ on each tetrahedron $K$. In \cite{chen2020discrete}, conforming and nonconforming  virtual element Gradgrad complexes are constructed on tetrahedral grids.

This paper considers the discretization of \eqref{eq:strongformEB} in the dual weak formulation of \cite{huliang2020} (see \eqref{eq:EinBian} below)  with $\mfieldE\in C^0([0,T],H(\ddiv\ddiv,\Omega;\Sbb))$,  $\sigma\in C^1([0,T],L^2(\Omega))$ and $\mfieldB\in C^0([0,T],H(\sym\ccurl;\Tbb))$. The associated finite element spaces for the linearized Einstein-Bianchi system is then closely related to the $\ddiv\ddiv$ complex    (dual to \eqref{eq:gradgradcomplex}) from \cite{arnold2021complexes,PaulyZ2020}\begin{align}\label{eq:complex3Dintro}\Scale[0.9]{
RT\xrightarrow[]{ \subset}H^1(\Omega;\R^3) \xrightarrow[]{ \ddev\nabla}H(\sym\ccurl,\Omega;\Tbb) \xrightarrow[]{\sym\ccurl}H(\ddiv\ddiv,\Omega;\Sbb) \xrightarrow[]{ \ddiv\ddiv}L^2(\Omega) \rightarrow 0.}
\end{align}
Such a  complex is exact provided that the domain is contractible and Lipschitz. A family of conforming finite element spaces    $\Sigma_{k,h}\subset H(\ddiv\ddiv,\Sbb)$ has been constructed for $k\geq 3$  in \cite{ChenHuang20203D}  on tetrahedral grids  (see \cite{ChenHuang20202D} in two dimensions). Another family of conforming finite element spaces $\widehat{\Sigma}_{k,h}\subset H(\ddiv\ddiv,\Sbb)$ has been designed in \cite{HuMaZhang2020} on simplicial grids for both two and three dimensions by using the $H(\ddiv,\Sbb)$ conforming finite element spaces from \cite{Hu2015trianghigh,HuZhang2014a,HuZhang2015tre}. The discontinuous Petrov-Galerkin (DPG) discretization of $H(\ddiv\ddiv,\Omega;\Sbb)$ can be found in \cite{FH2019,FHN2019}.
The main contribution of the paper is to construct a family of conforming  finite element spaces $\Lambda_{k+1,h}\subset H(\sym\ccurl;\Tbb) $ of order $k+1$ such that  the following finite element $\ddiv\ddiv$ complexes
\begin{align}\label{eq:complex3Dintrodis} {
RT\xrightarrow[]{ \subset}V_{k+2,h} \xrightarrow[]{ \ddev\nabla}\Lambda_{k+1,h} \xrightarrow[]{\sym\ccurl}\Sigma_{k,h} \xrightarrow[]{ \ddiv\ddiv}P_{k-1}(\T)\rightarrow 0 }
\end{align}
are exact together 
with a family of $H^1$ vectorial conforming finite element spaces $V_{k+2,h}$ of order $k+2$ and discontinuous piecewise  polynomial spaces $P_{k-1}(\T)$ on tetrahedral grids $\T$.    The construction is based on the finite element  de Rham complexes and the  new finite element strain complexes in two dimensions.  The associated finite elements of the strain complex  have extra continuity at vertices compared with those elements in \cite{ChenHuang20202D}. The construction shows some geometry decomposition in two and three dimensions as in \cite{CHK2018} for the de Rham complex, see more references therein.  These elements are used to discretize the linearized Einstein-Bianchi system, and the error estimates are provided in the end of the paper.  The polynomial degree of the lowest order  case for the pair $(\sigma,\mfieldE,\mfieldB)$ in this paper is  $P_1(K)-P_3(K;\Sbb)-P_4(K;\Tbb)$. It is more practical  compared with the finite element methods in \cite{huliang2020}.

 Throughout this paper, 
denote the space of all $3\times 3$ matrices by $\Mbb$, all symmetric $3\times 3$ matrices by $\Sbb$, and all  trace-free  $3\times 3$ matrices by $\Tbb$. Standard notation in Sobolev spaces will be used such as $L^2(\omega;X)$ and  $H^m(\omega,X)$, taking values in the finite-dimensional space $X$. Let $P_k(\omega;X)$ denote the set of all polynomials over $\omega$ of total degree not greater than $k$. The range space $X$ will be either $\R,\R^3,\mathbb{M},\Tbb,\Sbb$ in three dimensions.
If $X=\R$, then $L^2(\omega)$ abbreviates $L^2(\omega;X)$, similarly for $H^m(\omega)$ and $P_k(\omega)$.  

The organization of the paper is as follows. Section~\ref{sec:pre}
 introduces some operators for vectors and tensors.  Section~\ref{sec:FEM2D} introduces the  finite elements with extra continuity at vertices corresponding to the de Rham complex and strain complex, respectively. Section~\ref{sec:FEM3D} first constructs a family of $H(\sym\ccurl,\Tbb)$ conforming finite elements.  Second, the finite element $\ddiv\ddiv$ complexes \eqref{eq:complex3Dintrodis} are established  for a contractible domain. Section~\ref{sec:EB} uses the newly proposed finite element spaces to discretize  the linearized Einstein-Bianchi system within the dual formulation and shows the error estimates.
\section{Preliminaries}\label{sec:pre}
This section prepares some operators for vectors and tensors. More related results can be found in \cite{ChenHuang20203D}.  Let $\mfieldI$ denote the $3\times 3$ identity matrix.
Given a matrix $\mfieldA\in\R^{3\times 3
}$, define
\begin{align*}
\sym\mfieldA=\frac{1}{2}(\mfieldA+\mfieldA^T),\;\ddev \mfieldA=\mfieldA-\frac{1}{3}{\rm tr}(\mfieldA)\mfieldI.
\end{align*}
For a vector $\vfieldv=(v_1,v_2,v_3)^T$, define a skew-symmetric matrix as follows
\begin{align*}
{\rm mspn\ }\vfieldv=\begin{pmatrix}
0 &-v_3& v_2\\
v_3 &0& -v_1\\
-v_2&v_1&0 
\end{pmatrix}.
\end{align*}
For a matrix function $\mfieldA$,
the $\ccurl$ and $\ddiv$ operators apply row-wise to produce a matrix function $\ccurl \mfieldA$ and a vector function $\ddiv \mfieldA$, respectively. 
For a vector function $\vfieldv=(v_1,v_2,v_3)^T$, the gradient $\nabla$ applies row-wise to produce a matrix function 
\begin{align*}
\nabla \vfieldv=\begin{pmatrix}
\partial_x v_1 &\partial_y v_1 &\partial_z v_1\\
\partial_x v_2 &\partial_y v_2 &\partial_z v_2\\
\partial_x v_3 &\partial_y v_3&\partial_z v_3 
\end{pmatrix}.
\end{align*}
Define the symmetric gradient by
\begin{align*}
\epsilon (\vfieldv)=\sym (\nabla \vfieldv).
\end{align*}

Given a plane $f$ with the unit normal vector $\nvec$, for a vector $\vfieldv\in \R^3$, the following orthogonal decomposition holds
\begin{align*}
\vfieldv=\Pifn \vfieldv+\Pif \vfieldv:=(\vfieldv\cdot \nvec)\nvec+(\nvec\times \vfieldv)\times \nvec.
\end{align*}
Define the tangential derivatives and the surface curl operator  for a scalar function $	q$ as
\begin{align*}
\nabla_f q  =\Pif\nabla q=(\nvec\times \nabla q)\times \nvec,  \quad \ccurl_f q  =\nvec\times\nabla q.
\end{align*}
Define the surface rot operator for a vector function $\vfieldv$ as
\begin{align*}
\rrot_f \vfieldv  =(\nvec\times\nabla)\cdot \vfieldv=\nvec\cdot(\ccurl \vfieldv)
\end{align*}
and   the surface divergence $\ddiv_f$ as
\begin{align*}
\ddiv_f \vfieldv=\ddiv_f(\Pif \vfieldv)=(\nvec\times \nabla)\cdot(\nvec\times \vfieldv)=\rrot_f(\nvec\times \vfieldv).
\end{align*}
The surface gradient $\nabla_f\vfieldv$ applies row-wise to produce a matrix function $\nabla_f \vfieldv$.
 Define the surface symmetric gradient $\epsilon_f$  by
\begin{align*}
\epsilon_f(\vfieldv)=\sym(\nabla_f (\Pi_f \vfieldv)).
\end{align*} 
In particular, for $\nvec=(0,0,1)^T$, $f$ is the $x-y$ plane. Then, these operators $\nabla_f, \rrot_f,$ $\ccurl_f,\ddiv_f,\epsilon_f$ are standard differential operators in two dimensions. 

Given a matrix function $\mfieldA$,  $\nvec\times \mfieldA$ acts row-wise while $\mfieldA\times \nvec$ acts column-wise. The operators $\Pif$, $\rrot_f \mfieldA$ and $\ddiv_f \mfieldA$ act row-wise. Define the symmetric projection
\begin{align*}
\Pifsym\mfieldA=\sym(\Pif\mfieldA).
\end{align*}
Given a vector function $\vfieldv$ and a matrix function $\mfieldA$, the vector products commute with differentiation as follows
\begin{align*}
 (\nabla \vfieldv)^T\nvec=\nabla(\vfieldv\cdot\nvec),\\
 \nabla \vfieldv\times \nvec=\nabla(\vfieldv\times \nvec),\\
 (\ccurl \mfieldA)^T\nvec=\ccurl(\mfieldA^T\nvec).
\end{align*}

\section{Finite elements in two dimensions}\label{sec:FEM2D}
This section constructs new  finite elements on triangular grids in two dimensions for the de Rham complex and strain complex below with extra continuity at vertices. Some elements with extra continuity at vertices and along edges have been  introduced in \cite{CHK2018} for de Rham complexes in two and three dimensions. The finite elements  presented in this section have weaker continuity along edges compared  with  those elements from \cite{CHK2018}.  The finite elements for the strain complex have stronger continuity at vertices compared with those finite elements in \cite{ChenHuang20202D}. 
\subsection{Notation in two dimensions}
Throughout this section $\bx=(x,y)^T$ and $\bx^\perp=(-y,x)^T$. Suppose that $f$ is on the $x-y$ plane. Given a scalar function $q$ and a  vector function $\vfieldv=(v_1,v_2)^T$, the surface differential operators read $\nabla_fq=(\partial_x q,\partial_y q)^T,\ccurl_f q=(-\partial_yq,\partial_xq)^T$, $\ddiv_f\vfieldv=\partial_x v_1+\partial_y v_2$,  $\rrot_f \vfieldv=\partial_x v_2-\partial_y v_1$ and $\epsilon_f(\vfieldv)=\frac{1}{2}(\nabla_f\vfieldv+(\nabla_f\vfieldv)^T)$. To distinguish the  space $\Sbb$ of symmetric matrices in three dimensions, denote  the space of symmetric matrices  in two dimensions by $\Sbb_2:={\rm symmetric}\, \R^{2\times 2}$.

Let $RT_f$ denote the lowest order Raviart-Thomas space in two dimensions, which reads
\begin{align*}
    RT_f:=\{a+b\bx\ |\ a\in\R^2,b\in\R\}.
\end{align*}
Given a contractible domain $\omega\subset\R^2$, the de Rham complex  \cite{arnold2006}  in two dimensions reads 
\begin{align}\label{eq:complex2Da}
\R\xrightarrow[]{ \subset}H^1(\omega) \xrightarrow[]{ \nabla_f}H(\rrot_f,\omega;\R^2) \xrightarrow[]{\rrot_f} L^2(\omega) \rightarrow 0,
\end{align}
and the strain complex \cite{ChenHuang20202D} reads
\begin{align}\label{eq:complex2Db}
RT_f\xrightarrow[]{ \subset}H^1(\omega) \xrightarrow[]{ \epsilon_f}H(\rrot_f\rrot_f,\omega;\Sbb) \xrightarrow[]{\rrot_f\rrot_f} L^2(\omega) \rightarrow 0.
\end{align}
The corresponding polynomial complexes read \cite{arnold2006,ChenHuang20202D}  
\begin{align}\label{eq:polycomplex2Da}
\R\xrightarrow[]{ \subset}P_{k+2}(\omega)\xrightarrow[]{ \nabla_f}P_{k+1}(\omega;\R^2) \xrightarrow[]{\rrot_f}P_{k}(\omega)\rightarrow 0,\\\label{eq:polycomplex2Db}
RT_f\xrightarrow[]{ \subset}P_{k+2}(\omega;\R^2) \xrightarrow[]{ \epsilon_f}P_{k+1}(\omega;\Sbb) \xrightarrow[]{\rrot_f\rrot_f} P_{k-1}(\omega) \rightarrow 0.
\end{align}

Given a triangle $f$, let $\Ecal(f)$ denote the set of all edges of $f$. Given $e\in\Ecal(f)$, let $\nvec=(n_1,n_2)^T$ denote the unit normal vector of $e$ and $\tvec=(-n_2,n_1)^T$ denote the unit tangential vector of $e$. Let $\lambda_1,\lambda_2,\lambda_3$ denote the barycentric coordinates of $f$. 
\subsection{Finite element  de Rham complexes in two dimensions}This subsection constructs the finite elements with respect to the de Rham complex \eqref{eq:complex2Da}.

The shape function space of  the $H^1$ conforming  element is $P_{k+2}(f)$ with $k\geq 3$.  Define the set of polynomials of degrees $\leq k-1$ with vanishing values at vertices of $f$ by
\begin{align}\label{eq:secondchoiceBO}
\begin{aligned}
P_{k -1,0} (f):=\{p\in P_{k -1}(f)\ |\ p \text{ vanishes at all vertices of } f\}.
\end{aligned}
\end{align}
Following the idea  from \cite{huliang2020},
 the degrees of freedom of the $H^1$ conforming element with extra continuity at vertices are defined as follows
\begin{enumerate}
[label=(1\alph*)]
\item \label{enu:Honesa}
function value and first and second order derivatives at each vertex $\bx$:
\begin{align*}
p(\bx), \nabla_f p(\bx),\nabla^2_f p(\bx),
\end{align*}
\item\label{enu:Honesb}
 moments of order $\leq k-4$ on each edge $e$:
\begin{align*}
\int_e p q,\quad q\in P_{k-4}(e),
\end{align*}
\item\label{enu:Honesc} the interior degrees of freedom in $f$ defined by
\begin{align*}
\int_fpq,\quad q\in P_{k-1,0} (f).
\end{align*}
\end{enumerate} 
\begin{lemma}\label{lem:uniHone}
The degrees of freedom \ref{enu:Honesa}--\ref{enu:Honesc} are unisolvent for $P_{k +2}(f)$.\end{lemma}
\begin{proof}
It is easy to check that the number of the degrees of freedom is equal to the dimension of $P_{k +2}(f)$. Given $p\in P_{k +2}(f)$, if \ref{enu:Honesa}--\ref{enu:Honesb} vanish for $p$, then $p=0$ on $\partial f$. This shows that there exists some $r\in P_{k -1}(f)$ such that
\begin{align*}
p=\lambda_1\lambda_2\lambda_3 r.
\end{align*}  
Since the second order derivatives of $p$ vanish at each vertex of $f$, $r$ vanishes at each vertex of $f$ as well. 
The degrees of freedom \ref{enu:Honesc} with $P_{k -1,0} (f)$ from \eqref{eq:secondchoiceBO} lead to $r=0$.  This concludes the proof.
\end{proof}

\bigskip
The shape function space of the $H(\rrot_f)$ conforming element is 
$P_{k +1}(f;\R^2)$ with $k \geq 3$.  Before introducing its degrees of freedom, define the polynomial space
\begin{align}\label{eq:vecfacebsec}
\begin{aligned}
 P_{k-1 ,1} (f) :=  \{q\in P_{k -1}(f)\ |\ \text{ there exists some  }r\in P_{k -1,0} (f)\text{ such that }\ddiv_f(q\bx)=r \} .
\end{aligned}
\end{align}
The bijection  $\ddiv_f: P_{k-1}(f)\bx\rightarrow P_{k-1}(f)$ \cite{arnold2006} guarantees   $\dim  P_{k-1 ,1} (f)=\dim P_{k-1,0}(f)$.
The degrees of freedom of the $H(\rrot_f)$ conforming element with extra continuity at vertices then read
\begin{enumerate}
[label=(2\alph*)]
\item \label{enu:Hrotva}
function value and first  order derivatives at each vertex $\bx$:
\begin{align*}
\vfieldu(\bx), \nabla_f \vfieldu(\bx),
\end{align*}
\item\label{enu:Hrotvb}
 moments of order $\leq k-3$ of the tangential component on each edge $e$:
\begin{align*}
\int_e \vfieldu\cdot \tvec q,\quad q\in P_{k -3}(e),
\end{align*}
\item\label{enu:Hrotvc}
 moments of order $\leq k-2$ of $\rrot_f$ on each edge $e$:
\begin{align*}
\int_e \rrot_f \vfieldu q,\quad q\in P_{k -2}(e),
\end{align*}
\item\label{enu:Hrotvd} the interior degrees of freedom in $f$ defined by
\begin{align*}
\int_f\vfieldu\cdot \vfieldv,\quad \vfieldv\in  \ccurl_f P_{k-3}(f)+ P_{k -1,1} (f)\bx.
\end{align*}
\end{enumerate}

\begin{lemma}\label{lem:unirot}
The degrees of freedom \ref{enu:Hrotva}--\ref{enu:Hrotvc} are unisolvent for $P_{k +1}(f;\R^2)$.\end{lemma}
\begin{proof}
The number of the degrees of freedom is equal to the dimension of $P_{k +1}(f;\R^2)$, namely
\begin{align*}
18+3(k -2)+3(k -1)+k ^2-k -3=(k +2)(k +3).
\end{align*}  If  \ref{enu:Hrotva}--\ref{enu:Hrotvc}  vanish for some $\vfieldu\in P_{k +1}(f;\R^2)$, then $\vfieldu\cdot \tvec$ and $\rrot_f \vfieldu$ vanish on $\partial f$. This shows 
\begin{align*}
\rrot_f \vfieldu=\lambda_1\lambda_2\lambda_3 r
\end{align*}
for some $r\in P_{k -3}(f)$.  For any $q\in  P_{k-3}(f)$, an integration by parts  leads to
\begin{align*}
\int_f\vfieldu\cdot \ccurl_f q=-\int_f\lambda_1\lambda_2\lambda_3rq.
\end{align*} This  leads to $r=0$ provided that \ref{enu:Hrotvd} vanishes for $\vfieldu$ and thus $\rrot_f \vfieldu=0$. This and the zero tangential boundary conditions on $\partial f$ of $\vfieldu$ show  that there exists some $p\in P_{k -1}(f)$ such that
\begin{align*}
\vfieldu=\nabla_f (\lambda_1\lambda_2\lambda_3 p).
\end{align*}
Since the first order derivatives of $\vfieldu$ vanish at vertices, $p$ vanishes at the vertices and $p\in P_{k-1,0}(f)$. Given $\vfieldv\in P_{k-1,1}(f)\bx$,
\begin{align*}
\int_f \vfieldu\cdot \vfieldv=-\int_f \lambda_1\lambda_2\lambda_3p\ddiv_f\vfieldv.
\end{align*}
The degrees of freedom \ref{enu:Hrotvd} plus \eqref{eq:vecfacebsec} lead to $p=0$.
This concludes $\vfieldu=0$.
\end{proof}

\bigskip
The shape function space of the $L^2$ element is $P_{k}(f)$. The following degrees of freedom coincide with the Lagrange element of order $k$ as
\begin{enumerate}
[label=(3\alph*)]
\item \label{enu:Ltwosa}
function value  at each vertex $\bx$:
\begin{align*}
p(\bx),
\end{align*}
\item\label{enu:Ltwosb}
 moments  of order $\leq k-2$ on each edge $e$:
\begin{align*}
\int_e p q,\quad q\in P_{k -2}(e),
\end{align*}
\item\label{enu:Ltwosc} moments of order $\leq k-3$ in $f$:
\begin{align*}
\int_fpq,\quad q\in P_{k -3}(f).
\end{align*}
\end{enumerate} 

\bigskip
Let $B_{k +2,\nabla_f} (f)$ denote  the $H^1$ bubble function space of $P_{k +2}(f)$ with vanishing degrees of freedom \ref{enu:Honesa}--\ref{enu:Honesb}, let  $B_{k +1,\rrot_f} (f)$ denote the $H(\rrot_f)$ bubble function space of $P_{k +1}(f;\R^2)$ with vanishing degrees of freedom \ref{enu:Hrotva}--\ref{enu:Hrotvc}, and let $B_{k,0} (f)$ denote the $L^2$ bubble function space of $P_{k}(f)$ with vanishing degrees of freedom \ref{enu:Ltwosa}--\ref{enu:Ltwosb}. These bubble functions form  exact complexes as in the following lemma.
\begin{lemma} For $k\geq 3$,
it holds that
\begin{align*}
0\xrightarrow[]{ \subset}B_{k +2,\nabla_f} (f) \xrightarrow[]{ \nabla_f}B_{k +1,\rrot_f} (f)\xrightarrow[]{ \rrot_f} B_{k ,0} (f)/ P_0(f)\rightarrow 0.
\end{align*}
\end{lemma}
\begin{proof}
The proof of Lemma \ref{lem:unirot} leads to
$B_{k +1,\rrot_f}(f)\cap{\rm kerl}(\rrot_f)=\nabla_fB_{k +2,\nabla_f}(f)$.
This results in
\begin{align*}
{\rm dim}\rrot_fB_{k +1,\rrot_f}(f) &={\rm dim}B_{ k +1,\rrot_f,}(f)-{\rm dim} B_{k+2,\nabla_f}(f)  \\
&= \dim P_{k -3}(f)-1.
\end{align*}
 Since $B_{k ,0} (f)=\lambda_1\lambda_2\lambda_3P_{k -3}(f)$ and  the degrees of freedom \ref{enu:Hrotva}--\ref{enu:Hrotvc} imply
$ \rrot_f B_{k +1,\rrot_f} (f)\subseteq B_{k ,0} (f)/ P_0(f)$,
the combination with the previous identity concludes the proof.
\end{proof}

\bigskip
Given a bounded Lipschitz contractible polygonal domain $\omega\subset\R^2$, the finite elements with respect to the de Rham complex \eqref{eq:complex2Da} can be  constructed with  the above shape function spaces and degrees of freedom.  The   proof of the finite element de Rham complexes follows similar arguments as in \cite[Theorem 1]{CHK2018} and is omitted for brevity.
\subsection{Finite element  strain complexes in two dimensions}
\label{subsec:rotrot} This subsection constructs the finite elements with respect to the strain complex \eqref{eq:complex2Db}.

The shape function space  of   the $H^1$   vectorial conforming element  is
 $P_{k +2}(f;\R^2)$ with $k\geq 3$.  Recall the space $P_{k-1,0}(f)$ of polynomials degree $\leq k-1$ and  vanishing at vertices  of $f$ from \eqref{eq:secondchoiceBO}. 
Define $P_{k -1,0}(f;X):=\{\vfieldv\in P_{k-1}(f;X)\ |\ \text{each}$ $\text{component of $\vfieldv$ is in }P_{k-1,0}(f)\}$ for $X=\R^2$ or $\R^3$.
The degrees of freedom read
\begin{enumerate}
[label=(4\alph*)]
\item \label{enu:Honeva}
function value and first  order derivatives at each vertex $\bx$:
\begin{align*}
\vfieldu(\bx), \nabla_f \vfieldu(\bx),\nabla^2_f\vfieldu(\bx), 
\end{align*}
\item\label{enu:Honevb}
 moments of order $k-4$ of each component on each edge $e$:
\begin{align*}
\int_e \vfieldu\cdot \vfieldv,\quad \vfieldv\in P_{k -4}(e;\R^2),
\end{align*}
\item\label{enu:Honevc} the interior degrees of freedom in $f$ defined by
\begin{align*}
\int_f\vfieldu\cdot \vfieldv,\quad \vfieldv\in P_{k -1,0} (f;\R^2).
\end{align*}
\end{enumerate}

\begin{lemma}\label{lem:uniHonevec}
The degrees of freedom \ref{enu:Honeva}--\ref{enu:Honevc} are unisolvent for $P_{k +2}(f;\R^2)$.
\end{lemma}
\begin{proof}
The proof follows the same arguments as in   Lemma~\ref{lem:uniHone}.
\end{proof}

\bigskip
The shape function space  of the $H(\rrot_f\rrot_f,\Sbb_2)$ element is $P_{k +1}(f;\mathbb{S}_2)$
with $k \geq 3$.  Define
\begin{align}\label{eq:rotrotbub}
\begin{aligned}
P_{k -1,2}(f;\R^2):=\{&\vfieldv\in P_{k -1}(f;\R^2)\ |\ \text{there exists some } \vfieldw\in P_{k -1,0}(f;\R^2)\\
& \text{such that }\ddiv_f(   \sym(\vfieldv\bx^T))=\vfieldw \}.
\end{aligned}
\end{align} The bijection  $\ddiv_f: \sym(P_{k-1}(f;\R^2)\bx^T)\rightarrow P_{k-1}(f;\R^2)$ guarantees   $\dim  P_{k -1,2}(f;\R^2)=\dim P_{k-1,0}(f;\R^2)$ (see \cite[Lemma~3.6]{ChenHuang20202D} up to a rotation).
 The   degrees of freedom are given by 
\begin{enumerate}
[label=(5\alph*)]
\item \label{enu:Hrotrotta}
function value and first  order derivatives at each vertex $\bx$:
\begin{align*}
\mfieldt(\bx),\nabla_f\mfieldt(\bx),
\end{align*}
\item\label{enu:Hrotrottb}
 moments of order $\leq k-3$ of the tangential tangential component on each edge $e$:
\begin{align*}
\int_e \tvec^T\mfieldt\tvec q,\quad q\in P_{k -3}(e),
\end{align*}
\item\label{enu:Hrotrottc}
 moments   of order $\leq k-2$  of the following derivative on each edge $e$:
\begin{align*}
\int_e (-\partial_t(\nvec^T\mfieldt\tvec)+\tvec^T\rrot_f\mfieldt)  q,\quad q\in P_{k -2}(e),
\end{align*}
\item\label{enu:Hrotrottd} the interior degrees of freedom in $f$ defined by
\begin{align*}
\int_f\mfieldt:\mfieldxi,\quad \mfieldxi\in \ccurl_f \ccurl_f P_{k -1}(f)+\sym( P_{k -1,2} (f;\R^2)\bx^T).
\end{align*}
\end{enumerate}
\begin{remark}
The above finite element space  is a modified finite element space of $H(\rrot\rrot)$ as in \cite{ChenHuang20202D} with additional  continuity  at vertices.
\end{remark}
\begin{lemma}\label{lem:unisorotrot}
The degrees of freedom \ref{enu:Hrotrotta}--\ref{enu:Hrotrottd} are unisolvent for $P_{k +1}(f;\mathbb{S}_2)$.
\end{lemma}
 \begin{proof}A direction computation shows that the number of the degrees of freedom is equal to the dimension $P_{k +1}(f;\mathbb{S}_2)$ with
\begin{align*}
27+3(k -2)+3(k -1)+\frac{3 k  (k +1)}{2}-9=\frac{3(k +2)(k +3)}{2}.
\end{align*} Suppose that \ref{enu:Hrotrotta}--\ref{enu:Hrotrottd} vanish  for some $\mfieldt\in P_{k +1}(f;\mathbb{S}_2)$.
Similar arguments  as in \cite[Lemma~3.8]{ChenHuang20202D} show
 \begin{align*}
 \mfieldt=\epsilon_f(\lambda_1\lambda_2\lambda_3 \vfieldu)
 \end{align*}
 for some $\vfieldu\in P_{k -1}(f;\R^2)$. Since the first order derivatives of $\mfieldt$ vanish at each vertex, this shows  $\vfieldu\in P_{k-1,0}(f;\R^2)$.   For any $\mfieldxi\in \sym(P_{k-1,2}(f;\R^2)\bx^T)$, an integration by parts leads to
 \begin{align*}
 \int_f\mfieldt:\mfieldxi=-\int_f\lambda_1\lambda_2\lambda_3 \vfieldu\cdot\ddiv_f\mfieldxi.
 \end{align*}
The combination with vanishing  \ref{enu:Hrotrottd} and \eqref{eq:rotrotbub}  leads to $\vfieldu=0$. This concludes the proof.
 \end{proof}

\bigskip
Let $B_{k +2, \epsilon_f}(f)$ denote  the $H^1$ vectorial bubble function space of $P_{k+2}(f;\R^2)$ with vanishing degrees of freedom \ref{enu:Honeva}--\ref{enu:Honevc}. Let $B_{ k +1,\rrot_f\rrot_f}(f)$ denote  the $H(\rrot_f\rrot_f,\Sbb_2)$ bubble function space of $P_{k+1}(f;\Sbb_2)$ with vanishing degrees of freedom \ref{enu:Hrotrotta}--\ref{enu:Hrotrottc}. The bubble function spaces form   exact complexes  as in the following lemma.
\begin{lemma}\label{thm:facerotrotseq}Suppose $k \geq 3$.
The complexes
\begin{align*}
0\xrightarrow[]{ \subset}B_{k +2, \epsilon_f}(f) \xrightarrow[]{ \epsilon_f}B_{ k +1,\rrot_f\rrot_f}(f)\xrightarrow[]{ \rrot_f\rrot_f} P_{k -1}(f)/ P_1(f)\rightarrow 0 
\end{align*}
are exact.
\end{lemma}
\begin{proof}The proof 
follows similar arguments as in \cite[Lemma 3.9]{ChenHuang20202D}. Lemma \ref{lem:unisorotrot} shows $B_{ k +1,\rrot_f\rrot_f}(f)\cap{\rm ker}(\rrot_f\rrot_f)=\epsilon_f(B_{k +2, \epsilon_f}(f))$. This also means
\begin{align*}
\dim (\rrot_f\rrot_f B_{ k +1,\rrot_f\rrot_f}(f))=&\dim B_{ k +1,\rrot_f\rrot_f}(f)-\dim B_{k +2, \epsilon_f}(f)\\
=&\frac{1}{2}k (k +1)-3=\dim P_{k -1}(f)/ P_1(f).
\end{align*}
\end{proof}

Given a bounded Lipschitz contractible polygonal domain $\omega\subset\R^2$, the finite elements with respect to the strain complex \eqref{eq:complex2Db} can be  constructed with  the above shape function spaces and degrees of freedom.  The proof of  the  finite element strain complexes follows similar arguments as in \cite[Lemma 3.10]{ChenHuang20202D} and is omitted here. A rotation will lead to the finite element  $\ddiv\ddiv$ complexes in two dimensions which are modified ones from \cite{ChenHuang20202D}.
\section{Finite  element spaces in three dimensions}\label{sec:FEM3D}
This section constructs a family of   $H(\sym\ccurl,\Tbb)$ tensor conforming finite elements on tetrahedral  grids. The definitions of the degrees of freedom on faces will employ the polynomial spaces defined in Section~\ref{sec:FEM2D}. Although those spaces are defined when $f$ is on the $x-y$ plane, they can be extended to a three dimensional face $f$ with $\bx$ replaced by $\Pif\bx$ and $\bx^\perp$ replaced by $\nvec\times\bx$. It is proved that  the $H(\sym\ccurl,\Tbb)$ tensor conforming finite elements form  finite element $\ddiv\ddiv$ complexes  with the $H(\ddiv\ddiv,\Sbb)$ tensor conforming elements in \cite{ChenHuang20203D} and the newly constructed $H^1$ vectorial conforming elements in this paper.
\subsection{Further notation for three dimensions}
The space $RT$  in three dimensions reads as
\begin{align*}
    RT:=\{a+b\bx\ |\ a\in\R^3,b\in\R\}.
\end{align*}
Suppose that  $\Omega $ is a bounded, strong Lipschitz and contractible domain. Define
\begin{align*}
    H(\sym\ccurl,\Omega;\Tbb):=\{\mfieldt\in L^2(\Omega;\Tbb)\ |\ \sym\ccurl\mfieldt\in L^2(\Omega;\Sbb)\},\\
    H(\ddiv\ddiv,\Omega;\Sbb):=\{\mfieldt\in L^2(\Omega;\Sbb)\ |\ \ddiv\ddiv\mfieldt \in L^2(\Omega)\}.
\end{align*}
Recall the  $\ddiv\ddiv$ complex in three dimensions   \cite{arnold2021complexes,PaulyZ2020}   from \eqref{eq:complex3Dintro}
\begin{align}\label{eq:complex3D}\Scale[0.9]{
RT\xrightarrow[]{ \subset}H^1(\Omega;\R^3) \xrightarrow[]{ \ddev\nabla}H(\sym\ccurl,\Omega;\Tbb) \xrightarrow[]{\sym\ccurl}H(\ddiv\ddiv,\Omega;\Sbb) \xrightarrow[]{ \ddiv\ddiv}L^2(\Omega) \rightarrow 0,}
\end{align}
and the  polynomial complex \cite{ChenHuang20203D}  
\begin{align}\label{eq:complex3Dpol}
RT\xrightarrow[]{ \subset}P_{k+2}(\Omega;\R^3) \xrightarrow[]{ \ddev\nabla}P_{k+1}(\Omega;\Tbb)\xrightarrow[]{\sym\ccurl}P_{k}(\Omega;\Sbb) \xrightarrow[]{ \ddiv\ddiv}P_{k-2}(\Omega)\rightarrow 0.
\end{align}

Let $\T$ be a shape regular triangulation of $\Omega\subset\R^3$ into tetrahedra. Let $\Vcal$ denote the set of all vertices, $\Ecal$ the set of all edges and $\Fcal$ the set of all faces. Given $e\in\Ecal$,   let $\tvec$ denote the unit tangential vector along $e$, and let $\nvec_1$ and  $\nvec_2$ denote two independent unit normal vectors such that $\nvec_1\times \nvec_2=\tvec$. Given $f\in\Fcal$, let $\nvec$ denote the unit normal vector of $f$,  and let  $\nvec_{\partial f}$ denote the outnormal vector of $\partial f$ on $f$ and $\tvec_{\partial f}$ denote the unit tangential vector of $\partial f$ such that $\nvec_{\partial f}\times \tvec_{\partial f}=\nvec$.  Given $K\in\T$, let $\lambda_j$ with $1\leq j\leq 4$ denote the barycentric coordinates of $K$.


\subsection{$H(\sym\ccurl,\Tbb)$ conforming finite elements}\label{sec:curl}Recall the space $\Tbb=\{\mfieldt\in\R^{3\times 3}\ |\ \tr(\mfieldt)=0\}$ of trace free matrices.  The shape function space of  the $H(\sym\ccurl,\Tbb)$ element  is  $P_{k +1}(K;\Tbb)$    with $k \geq 3$. Recall the   operators $\Pif$ and $\Pifsym$  from Section \ref{sec:pre} for face $f\in\Fcal$ with the unit normal vector $\nvec$. The degrees of freedom on the faces require the polynomial spaces $P_{k-1,0}(f)$ from \eqref{eq:secondchoiceBO}, $P_{k-1,1}(f)$ from  \eqref{eq:vecfacebsec} and $P_{k -1,2}(f;\R^2)$ from \eqref{eq:rotrotbub}, which are extended  to  three dimensional faces as follows 
\begin{align*}
\begin{aligned}
 P_{k-1 ,1} (f) :=  \{q\in P_{k -1}(f)\ |\ \text{there exists  some }r\in P_{k -1,0} (f)\text{ such that }\ddiv_f(q\Pif\bx)=r \} ,
\end{aligned}
\end{align*}
and 
\begin{align*}
P_{k -1,2}(f;\R^2):=&\{\vfieldv\in \Pif P_{k -1}(f;\R^3)\ |\ \text{there exists  some }\vfieldw\in\Pif P_{k -1,0}(f;\R^3)\\
&\text{ such that }\ddiv_f(   \sym(\vfieldv(\Pif\bx)^T))=\vfieldw\}.
\end{align*}

The degrees of freedom of the $H(\sym\ccurl,\Tbb)$ element are defined as follows
\begin{enumerate}[label=(7\alph*)]
\item \label{enu:HScurla} function value and  first  order derivatives of each component at each vertex $\bx\in\Vcal$:
\begin{align*}
\mfieldt(\bx),\nabla\mfieldt (\bx) ,
\end{align*}
\item\label{enu:HScurlb} moments of order $\leq k-3$ of the following components on each edge $e\in \Ecal$:
\begin{align*}
\int_e \nvec_{i}^T \mfieldt\tvec q, \quad q\in P_{k -3}(e), i=1,2,
\end{align*}
\item\label{enu:HScurlc} moments of order $\leq k-2$ of the following  derivatives on each edge $e\in \Ecal$:
\begin{align*}
\int_e \nvec_{i}^T\sym\ccurl\mfieldt  \nvec_{j} q,\quad q\in P_{k -2}(e),i,j=1,2,\\
\int_e\big(\nvec_{1}^T\ccurl\mfieldt \nvec_{2} -\partial_t(\tvec^T\mfieldt\tvec)\big)q,\quad q\in P_{k -2}(e),
\end{align*}
\item \label{enu:HScurld} degrees of freedom on each face $f\in \Fcal$ defined by
\begin{align*}
\int_f\Pif (\mfieldt^T \nvec)\cdot  \vfieldv,\quad \vfieldv\in \ccurl_f P_{k-3}(f)+P_{k -1,1}(f)\Pif\bx,
\end{align*}
\item \label{enu:HScurle} degrees of freedom on each face $f\in \Fcal$ defined by
\begin{align*}
\int_f   \Pifsym(\mfieldt\times \nvec) :\mfieldxi,\quad \mfieldxi\in\ccurl_f \ccurl_f P_{k -1}(f)+\sym( P_{k -1,2}(f;\R^2)(\Pif\bx)^T),
\end{align*}
\item\label{enu:HScurlf} interior degrees of freedom in each element $K\in\Tcal$ defined by
\begin{align*}
&\int_K \sym\ccurl\mfieldt: \mfieldxi, \quad\mfieldxi\in  \sym(\bx\times P_{k -2}(K;\Tbb)),\\
&\int_{f_1}\sym\ccurl\mfieldt\nvec\cdot \vfieldv,\quad \vfieldv\in P_{k -2}(f_1)(\nvec\times\bx) \text{ for an arbitrarily but fixed face }f_1,\\
&\int_{K}\mfieldt:\mfieldxi,\quad \mfieldxi\in\ddev(P_{k -2}(K;\R^3)
\bx^T).
\end{align*}
\end{enumerate}
\begin{remark}
The first degrees of freedom in \ref{enu:HScurlc} plus \ref{enu:HScurla} imply the continuity of $\nvec_i^T\sym\ccurl\mfieldt\nvec_j^T$. This, the second degrees of freedom   in \ref{enu:HScurlc} plus \ref{enu:HScurla} imply  the continuity of $\nvec_{1}^T\ccurl \mfieldt \nvec_{2} -\partial_t(\tvec^T\mfieldt \tvec)$ and $\nvec_{2}^T\ccurl \mfieldt \nvec_{1} +\partial_t(\tvec^T\mfieldt \tvec)$. Besides, they are independent of the choices of the normal vectors. Suppose that there are another two unit normal vectors $\nvec_1^\prime=c_1\nvec_1+c_2\nvec_2$ and $\nvec_2^\prime=-c_2\nvec_1+c_1\nvec_2$ such that $\nvec_1^\prime\times\nvec_2^\prime=\tvec$. An elementary  computation leads to the continuity of $(\nvec_i^\prime)^T\sym\ccurl \mfieldt \nvec_{j}^\prime$. The continuity of $(\nvec_1^\prime)^T\ccurl\mfieldt  \nvec_2^\prime-\partial_t(\tvec^T\mfieldt \tvec)$ follows from
\begin{align*}
(\nvec_1^\prime)^T\ccurl\mfieldt  \nvec_2^\prime-\partial_t(\tvec^T\mfieldt \tvec)=&-c_1c_2\nvec_1^T\ccurl\mfieldt  \nvec_1+c_1c_2\nvec_2^T\ccurl\mfieldt \nvec_2\\
&+c_1^2\nvec_1^T\ccurl\mfieldt  \nvec_2-c_2^2 \nvec_2^T\ccurl\mfieldt  \nvec_1-(c_1^2+c_2^2)\partial_t(\tvec^T\mfieldt \tvec).
\end{align*} The first two degrees of freedom  in \ref{enu:HScurlf} are motivated from those of the $H(\ddiv\ddiv,\Sbb)$ conforming finite elements in \cite{ChenHuang20203D}.
\end{remark}

The proof of unisolvence of the degrees of freedom \ref{enu:HScurla}-\ref{enu:HScurlf}  requires the following  three lemmas. The identities in the first two lemmas present  the restrictions of functions and operators on faces and show some connections with the finite elements in two dimensions from Section \ref{sec:FEM2D}. The third lemma from \cite{chen2020discrete} will be used to deal with interior degrees of freedom.
\begin{lemma}\label{lem:facetoedge}
Given $f$ with two unit tangential  vectors $\tvec_1$ and $\tvec_2$ such that $\tvec_{1}\times \tvec_{2} =\nvec$,  it holds that
\begin{align}\label{eq:restone}
\Pif(\mfieldt^T\nvec)\cdot\tvec_2=\nvec^T\mfieldt\tvec_2,\\\label{eq:resttwo}
\rrot_f\Pif(\mfieldt^T\nvec)=\nvec^T\ccurl\mfieldt\nvec,\\\label{eq:restthree}
\tvec_{2}^T \Pifsym(\mfieldt\times\nvec)\tvec_{2}=-\tvec_{1}^T\mfieldt\tvec_2,\\\label{eq:restfour}
-\partial_{t_2}(\tvec_1^T\Pifsym(\mfieldt\times\nvec) \tvec_{2})\hspace{-0.5mm}+\hspace{-0.5mm}\tvec^T_{2}\rrot_f \Pifsym(\mfieldt\times\nvec)
=& -\tvec_{1}^T\ccurl\mfieldt \nvec\hspace{-0.5mm}-\partial_{t_{2}}\big( \tvec_{2}^T \mfieldt \tvec_{2}). \hspace{-0.5mm}
\end{align}
\end{lemma}
\begin{proof}
Let $\vfieldw =\Pif(\mfieldt^T\nvec)$. A direct computation shows
\begin{align*}
    \vfieldw\cdot\tvec_2=\mfieldt^T\nvec\cdot\tvec_2=\nvec^T\mfieldt\tvec_2,\\
    \rrot_f \vfieldw=\nvec\cdot\ccurl(\mfieldt^T\nvec)=\nvec^T\ccurl\mfieldt\nvec.
\end{align*}
This proves \eqref{eq:restone}--\eqref{eq:resttwo}.
Let $\mfieldzeta =\Pi_{f,\sym}(\mfieldt\times\nvec)$.
The cross product rule plus $\tvec_{1}\times \tvec_{2}=\nvec$ show
\begin{align*}
\tvec_{2}^T\mfieldzeta\tvec_{2}=\tvec_{2}^T(\mfieldt\times \nvec )\tvec_{2}=-\tvec_{1}^T \mfieldt \tvec_{2}.
\end{align*}
This proves \eqref{eq:restthree}.
Some elementary computations lead to 
\begin{align*}
\tvec_{1}^T\mfieldzeta \tvec_{2}&=\frac{1}{2}(\tvec_{2}^T \mfieldt \tvec_{2}-\tvec_{1}^T\mfieldt\tvec_1),\\
\tvec_{2}^T\rrot_f\big( \Pif(\mfieldt\times\nvec) \big)&=\rrot_f\big((\Pif(\mfieldt\times\nvec))^T\tvec_{2}\big) 
\\&=-\rrot_f( \mfieldt^T\tvec_{1})=-\tvec_{1}^T\ccurl\mfieldt\nvec,\\
\tvec_{2}^T\rrot_f\big(( \Pif(\mfieldt\times\nvec))^T \big)&=\rrot_f\big( \Pif(\mfieldt\times\nvec)) \tvec_{2}\big)=\rrot_f(\mfieldt \tvec_{2}\times\nvec)\\
&=-\ddiv_f(\mfieldt\tvec_{2})=-\partial_{t_{2}}(\tvec_{2}^T \mfieldt \tvec_{2})-\partial_{t_1}(\tvec_{1}^T \mfieldt \tvec_{2}).
\end{align*}
The previous three identities plus $\rrot_f(\mfieldt ^T\tvec_{1})=-\partial_{t_{2}}(\tvec_{1}^T\mfieldt\tvec_{1})+\partial_{t_1}(\tvec_{1}^T \mfieldt \tvec_{2})$ lead to
\begin{align*}
-\partial_{t_2}(\tvec_{1}^T\mfieldzeta \tvec_{2})+\tvec^T_{2}\rrot_f \mfieldzeta
=& -\tvec_{1}^T\ccurl\mfieldt \nvec-\partial_{t_{2}}\big( \tvec_{2}^T \mfieldt \tvec_{2}) .
\end{align*}
This proves  \eqref{eq:restfour}.
\end{proof}
\begin{lemma}\label{lem:symcurlrestri}Let $f\in\Fcal$ with the unit normal vector $\nvec$. (a)
Suppose  $\mfieldxi=\sym\ccurl\mfieldt$. Then,
\begin{align}\label{eq:facecurlone}
\nvec^T\mfieldxi \nvec=\rrot_f(\Pif(\mfieldt^T \nvec)),\\
\label{eq:facecurltwo}
2\ddiv_f(\mfieldxi \nvec)+\partial_n(
\nvec^T\mfieldxi\nvec)=-\rrot_f\rrot_f\Pifsym(\mfieldt\times \nvec).
\end{align}
(b) Suppose $\mfieldt=\ddev\nabla \vfieldv$. Then,
\begin{align}\label{eq:facegradone}
\Pif(\mfieldt^T \nvec)=\nabla_f(\vfieldv\cdot \nvec),\\
\label{eq:facegradtwo}
\Pifsym(\mfieldt\times \nvec)=\epsilon_f(\vfieldv\times \nvec).
\end{align}
\end{lemma}
\begin{proof}{\em Proof of (a)}.
The first identity \eqref{eq:facecurlone} follows from 
the definition of $\rrot_f$ with
 \begin{align*}
 \nvec^T\ccurl \mfieldt \nvec =\ccurl(\mfieldt^T\nvec)\cdot \nvec=\rrot_f(\mfieldt^T \nvec)=\rrot_f(\Pif(\mfieldt^T\nvec)).
 \end{align*}
 The second term on the left-hand side of \eqref{eq:facecurltwo} satisfies
 \begin{align*}
 \partial_n(\nvec^T\ccurl \mfieldt \nvec)&= \partial_n\big(\rrot_f(\mfieldt^T \nvec)\big)= \rrot_f\big(\partial_n(\mfieldt^T\nvec)\big).
 \end{align*}
For $\ccurl\mfieldt \nvec$,
 \begin{align*}
\ddiv_f( \ccurl \mfieldt \nvec)&= \rrot_f(\nvec\times \rrot_f\mfieldt) = -\rrot_f\rrot_f( \mfieldt\times \nvec).
 \end{align*}
As for $( \ccurl \mfieldt)^T\nvec=\ccurl(\mfieldt^T\nvec)$, the cross product rule   leads to
 \begin{align*}
\ddiv_f\big( (\ccurl \mfieldt)^T\nvec\big)&=\rrot_f\big(\nvec
\times\ccurl(\mfieldt^T \nvec)\big)\\
&=\rrot_f\big(\nabla(\nvec^T\mfieldt \nvec)-\partial_n( \mfieldt^T\nvec)\big)=- \rrot_f\big(\partial_n(\mfieldt^T\nvec)\big).
 \end{align*}
The previous arguments lead to
\begin{align*}
2\ddiv_f(\xi \nvec)+\partial_n(\nvec^T\xi \nvec)=-\rrot_f\rrot_f( \mfieldt\times \nvec).
\end{align*}
Since
 \begin{align*}
\rrot_f\rrot_f( \mfieldt\times \nvec)&=\rrot_f\rrot_f( \Pif(\mfieldt\times \nvec))=\rrot_f\rrot_f\big((\Pif(\mfieldt\times \nvec))^T\big),
 \end{align*}
 the combination with the previous identity concludes \eqref{eq:facecurltwo}.
 
 {\em Proof of (b)}. Since $\Pif \nvec=0$, it holds that
\begin{align*}
\Pif(\mfieldt^T \nvec)=\Pif\big((\nabla \vfieldv)^T\nvec\big)=\Pif(\nabla(\vfieldv\cdot \nvec))=\nabla_f(\vfieldv\cdot \nvec).
\end{align*}
This proves \eqref{eq:facegradone}. Since $\mfieldI\times \nvec=\nvec\times \mfieldI=-(\nvec\times \mfieldI)^T$, this leads to
\begin{align*}
\Pi_{f,\sym}(\mfieldt\times \nvec)=\Pifsym(\nabla \vfieldv\times \nvec)=\Pifsym(\nabla(\vfieldv\times \nvec))=\epsilon_f(\vfieldv\times \nvec).
\end{align*}
This proves \eqref{eq:facegradtwo}.
\end{proof}

\begin{lemma}[\cite{chen2020discrete}]\label{lem:decomtwo}
It holds that
\begin{align*}
    \ddiv:\ddev(  P_{k}(\Omega;\R^3)\bx^T )\rightarrow P_{k}(\Omega;\R^3) .
\end{align*}
is a bijection. 
\end{lemma}

\begin{theorem}\label{thm:Hsymcurluni}
The degrees of freedom \ref{enu:HScurla}--\ref{enu:HScurlf} are unisolvent for $P_{k +1}(K;\Tbb)$.
\end{theorem}
\begin{proof}
Note that $\dim( \sym(\bx\times P_{k -2}(K;\Tbb))=\frac{k (k -1)(5k +14)}{6}$  from \cite[Lemma 4.6]{ChenHuang20203D}.
 A direct computation shows that the number of the degrees of freedom is equal to the dimension of $P_{k +1}(K;\Tbb)$, namely
\begin{align*}
128&+12(k -2)+24(k -1) +4(k ^2-k -3)+4(\frac{3k (k +1)}{2}-9)\\
&+\frac{k (k -1)(5k +14)}{6}+\frac{k (k -1)}{2}+\frac{(k -1)k (k +1)}{2} = \frac{4(k +2)(k +3)(k +4)}{3} .
\end{align*}
It suffices to prove if \ref{enu:HScurla}--\ref{enu:HScurle} vanish  for $\mfieldt\in P_{k +1}(K;\Tbb)$ then $\mfieldt=0$. The degrees of freedom \ref{enu:HScurla}--\ref{enu:HScurlb} show,  on each edge $e\in \Ecal(K)$,
\begin{align}
\label{eq:edgecurlvalue}
 \nvec_{i}^T\mfieldt \tvec=0,\quad  i=1,2.
\end{align}

Given $f\in\Fcal(K)$,  let   $\vfieldw =\Pif(\mfieldt^T\nvec)$  and $\mfieldzeta =\Pifsym(\mfieldt\times\nvec)$.  Lemma~\ref{lem:facetoedge} with $\tvec_1=\nvec_{\partial f}$ and $\tvec_2=\tvec_{\partial f}$ shows
\begin{align*}
\vfieldw\cdot \tvec_{\partial f}=& \nvec^T\mfieldt \tvec_{\partial f}\text{ and }\rrot_f\vfieldw=\nvec^T\ccurl \mfieldt \nvec=\nvec^T\sym\ccurl\mfieldt\nvec.
\end{align*}
The combination with \eqref{eq:edgecurlvalue}, \ref{enu:HScurla}, \ref{enu:HScurlb} and   the first degrees of freedom in \ref{enu:HScurlc} leads to
\begin{align}
\label{eq:facecurlvalue}
\vfieldw\cdot \tvec_{\partial f}=0 \text{ and }\rrot_f \vfieldw=0\text{ on }\partial f. 
\end{align} 
The degrees of freedom \ref{enu:HScurld} combined with similar arguments as in Lemma~\ref{lem:unirot} for $f$ on the $x-y$ plane lead to $\vfieldw=0$ on $f$.
On the other hand, \eqref{eq:restthree}--\eqref{eq:restfour} show
\begin{align*}
\tvec_{\partial f}^T\mfieldzeta \tvec_{\partial f}=-\nvec_{\partial f}^T \mfieldt \tvec_{\partial f},\\
-\partial_{t_{\partial f}}(\nvec_{\partial f}^T\mfieldzeta \tvec_{\partial f})+\tvec^T_{\partial f}\rrot_f \mfieldzeta= -\nvec_{\partial f}^T\ccurl\mfieldt \nvec-\partial_{t_{\partial f}}\big( \tvec_{\partial f}^T \mfieldt \tvec_{\partial f}) .
\end{align*}
Hence the former identity plus \ref{enu:HScurla} and \ref{enu:HScurlb} show $\tvec_{\partial f}^T\mfieldzeta \tvec_{\partial f}=0$, and the latter identity plus \ref{enu:HScurla} and \ref{enu:HScurlc} show
$-\partial_{t_{\partial f}}(\nvec_{\partial f}^T\mfieldzeta \tvec_{\partial f})+\tvec^T_{\partial f}\rrot_f \mfieldzeta=0$ on $\partial f$. The degrees of freedom \ref{enu:HScurle} combined with similar arguments  as in Lemma~\ref{lem:unisorotrot} lead to $\mfieldzeta=0$. The cross product rules show   
\begin{align*}
(\nvec\times \mfieldt+(\nvec\times \mfieldt)^T)\nvec =\nvec\times(\mfieldt^T\nvec) ,\\
\nvec\times (\nvec\times \mfieldt+(\nvec\times \mfieldt)^T)\times \nvec=-\Pif(\mfieldt\times\nvec)-(\Pif(\mfieldt\times\nvec))^T.
\end{align*}
The combination with $\vfieldw=0$ and $\mfieldzeta=0$ implies
\begin{align}\label{eq:facecont}
\nvec\times \mfieldt+(\nvec\times \mfieldt)^T=0 \text{ on }f.
\end{align}

Let $\mfieldxi =\sym\ccurl\mfieldt$. For any $f\in\Fcal(K)$ with the unit normal vector $\nvec$, it follows from Lemma~\ref{lem:symcurlrestri} that
\begin{align*}
\nvec^T\mfieldxi \nvec=\rrot_f(\Pif(\mfieldt^T \nvec))=0,\\
2\ddiv_f(\mfieldxi \nvec)+\partial_n(
\nvec^T\mfieldxi\nvec)=-\rrot_f\rrot_f\Pifsym(\mfieldt\times \nvec)=0.
\end{align*}
For any $e\in\Ecal(K)$,
\begin{align*}
    \nvec_i^T\mfieldxi\nvec_j=\nvec_i^T\sym\ccurl\mfieldt\nvec_j=0,\;i,j=1,2.
\end{align*}
Hence $\mfieldxi$ is a $H(\ddiv\ddiv)$ bubble function  and $\ddiv\ddiv\mfieldxi=0$. The first two degrees of freedom in \ref{enu:HScurlf} and similar arguments as in \cite[Lemma~4.6]{ChenHuang20203D} lead to $\mfieldxi=0$.  Then the polynomial complex \eqref{eq:complex3Dpol} shows that there exists $\vfieldu\in P_{k+2}(K;\R^3) $ such that
\begin{align}\label{eq:curlgradproof}
\mfieldt=\ddev\nabla \vfieldu\text{ and }Q_0^f(\vfieldu\cdot \nvec)=0\text{ for all }f\in \Fcal(K) 
\end{align}
with the $L^2$ projection $Q_0^f$ onto $P_0(f)$.

 Given $f\in\Fcal(K)$ with the unit normal vector $\nvec$,  the boundary conditions $\Pif(\mfieldt^T\nvec)=0$ and $\Pifsym(\mfieldt\times\nvec)=0$ for $\mfieldt$, Lemma \ref{lem:symcurlrestri} and \eqref{eq:curlgradproof} lead to  $\vfieldu\cdot \nvec=0$, and $\epsilon_f(\vfieldu\times \nvec)=0$ on $f$. This shows $\vfieldu=0$ at each vertex of $K$ and $\vfieldu\times \nvec$ is a linear function on $f$.  Furthermore, $\vfieldu\times \nvec=0$ on $f$ and thus $\vfieldu=0$ on $\partial K$.  This leads to  the existence of some $\vfieldv\in P_{k -2}(K;\R^3)$  with
\begin{align*}
    \vfieldu=\lambda_1\lambda_2\lambda_3\lambda_4\vfieldv.
\end{align*}
For any $\mfieldchi\in\ddev(P_{k-2}(K;\R^3)\bx^T)$, this and an integration by parts lead to
\begin{align*}
\int_K\mfieldt:\mfieldchi=-\int_K\lambda_1\lambda_2\lambda_3\lambda_4\vfieldv\cdot\ddiv\mfieldchi.
\end{align*}
The combination with the bijection in Lemma~\ref{lem:decomtwo} and the third vanishing degrees of freedom in \ref{enu:HScurlf} implies
$\vfieldv=0$ and hence $\mfieldt=0$. This concludes the proof.
\end{proof}  
\begin{remark} 
Let $\tvec_{1}$ and $\tvec_{2}$ denote two independent unit tangential vectors of $f$. Define the  space of traceless matrices related to the face $f$ as follows
\begin{align*}\label{eq:bubbleelem}
\Tbb_f:={\rm span}\{\tvec_{1}\nvec^T,\tvec_{2}\nvec^T,\nvec\nvec^T-\frac{1}{3}I\}.
\end{align*}  Then the $H(\sym\ccurl,\Tbb)$ bubble function space with respect to the degrees of freedom  \ref{enu:HScurla}--\ref{enu:HScurle} on $K$ reads
\begin{align}\label{eq:bubsymcurl}
\begin{aligned}
B_{ k +1,\sym\ccurl}(K):=&\lambda_{1}\lambda_2\lambda_3\lambda_4P_{k -3}(K;\Tbb)+\sum_{f\in \Fcal(K)}\lambda_{f,1}\lambda_{f,2}\lambda_{f,3}P_{k -2}(f)\Tbb_f 
\end{aligned}
\end{align}
 with the barycentric coordinates $\lambda_{f,1}, \lambda_{f,1}, \lambda_{f,3}$ with respect to $f$.
In fact,
it follows from \eqref{eq:facecont} that
$\mfieldt=0\text{ on all  edges }e\in\Ecal(K) $. Then  
\begin{align*}
\mfieldt\in \sum_{1\leq i<j<k\leq 4}\lambda_{i}\lambda_{j}\lambda_{k} P_{k -2}(K;\Tbb) .
\end{align*}
If $\mfieldt\neq 0$ on some  $f$, then $\mfieldt|_f\in \lambda_{f,1}\lambda_{f,2}\lambda_{f,3}P_{k -2}(f;\Tbb)$.
Since  $ \nvec\times \mfieldt+(\nvec\times \mfieldt)^T|_f=0$,
this leads to $\mfieldt|_f\in \lambda_{f,1}\lambda_{f,2}\lambda_{f,3} P_{k-3}(f)\Tbb_f$. If $\mfieldt=0$ on all faces $f\in \Fcal(K)$, then $\mfieldt\in\lambda_1\lambda_2\lambda_3\lambda_4 P_{k-4}(K;\Tbb)$.
\end{remark}

\bigskip
The proof of \eqref{eq:facecont}  in Theorem \ref{thm:Hsymcurluni} implies $(\nvec\times \mfieldt)+(\nvec\times \mfieldt)^T$ is continuous across faces if the degrees of freedom \ref{enu:HScurla}--\ref{enu:HScurlf} are single-valued. 
This allows the definition of the following   $H(\sym\ccurl,\Tbb)$  conforming finite element  space $\Lambda_{k+1,h}$ with $k\geq 3$ by
\begin{align}
\label{eq:symcurlspace}
\begin{aligned}
\Lambda_{ k +1,h}:=&\{ \mfieldt_h\in H {(\sym\ccurl,\Omega;\Tbb)}\ |\ \mfieldt_h|_K\in P_{k +1}(K;\Tbb )\text{ for all }K\in\Tcal,\\
&\text{ all the degrees of freedom \ref{enu:HScurla}--\ref{enu:HScurlf} are single-valued}\}.
\end{aligned}
\end{align}
\subsection{Finite element $\ddiv\ddiv$ complexes}
Recall the $H(\ddiv\ddiv)$ finite element spaces   from \cite{ChenHuang20203D}. The shape function space is $ P_{k }(K;\Sbb) $ with $k \geq 3$ and the degrees of freedom are defined by
\begin{enumerate}[label=(8\alph*)]
\item\label{enu:Hdivdiva}function value  at each vertex $\bx\in\Vcal$:
\begin{align*}
\mfieldt(\bx),
\end{align*}
\item\label{enu:Hdivdivb}moments of order $\leq k-2$ of the following components on each edge $e\in \Ecal$:
\begin{align*}
\int_e \nvec_i^T\mfieldt \nvec_jq,\quad q\in P_{k -2}(e), i,j=1,2,
\end{align*}
\item\label{enu:Hdivdivc}moments of order $\leq k-3$ of the normal normal component on each face $f\in \Fcal$:
\begin{align*}
\int_f \nvec^T\mfieldt \nvec  q,\quad q\in P_{k -3}(f),
\end{align*}
\item\label{enu:Hdivdivd}moments of order $\leq k-1$ of the following derivative  on each face $f\in \Fcal$:
\begin{align*}
\int_f(2\ddiv_f(\mfieldt \nvec)+\partial_n(\nvec^T\mfieldt \nvec))q,\quad q\in P_{k -1}(f),
\end{align*}
\item\label{enu:Hdivdive}interior degrees of freedom in each element $K\in \Tcal$ defined by
\begin{align*}
\int_K\mfieldt:\mfieldxi,\quad \mfieldxi \in \nabla^2 P_{k -2}(K)+\sym(\bx\times P_{k -2}(K;\Tbb)),
\end{align*}
\item\label{enu:Hdivdivf}interior degrees of freedom in each element $K\in\Tcal$ defined by
\begin{align*}
\int_{f_1} \mfieldt \nvec\cdot  \vfieldv,\quad \vfieldv\in P_{k -2}(f_1)(\nvec\times\bx) \text{ for an arbitrarily but fixed face }f_1.
\end{align*}
\end{enumerate}

The  $H(\ddiv\ddiv,\Sbb)$  conforming finite element space $\Sigma_{k ,h}$ is then defined by
\begin{align}
\label{eq:divdivspace}
\begin{aligned}
\Sigma_{k ,h}:=&\{ \mfieldt_h\in H (\ddiv\ddiv,\Omega;\Sbb)\ |\ \mfieldt_h|_K\in P_{k }(K;\Sbb )\text{ for all }K\in\Tcal,\\
&\text{ all the degrees of freedom \ref{enu:Hdivdiva}--\ref{enu:Hdivdivf} are single-valued}\}.
\end{aligned}
\end{align}

For the conforming finite element space of $H^1(\Omega;\R^3)$,
the shape function space is $ P_{k +2}(K;\R^3)$.  
Recall the space  $P_{k -1,0}(f)$ from \eqref{eq:secondchoiceBO}.  
The degrees of freedom   are defined as follows
\begin{enumerate}[label=(9\alph*)]
\item  \label{enu:Honea}function value and  first and second order derivatives of each component at each vertex $\bx\in\Vcal$:
\begin{align*}
\vfieldu(\bx),\nabla \vfieldu(\bx), \nabla^2 \vfieldu (\bx), 
\end{align*}
\item  \label{enu:Honeb}moments of order $\leq k-4$ of each component on each edge $e\in\Ecal$:
\begin{align*}
\int_e \vfieldu\cdot  \vfieldv,\quad \vfieldv\in P_{k -4}(e),\,i=1,2,
\end{align*}
\item \label{enu:Honec}degrees of freedom  on each face $f\in\Fcal$ defined by
\begin{align*}
\int_f \vfieldu\cdot \vfieldv\quad \vfieldv\in P_{k -1,0}(f;\R^3),
\end{align*}
\item \label{enu:Honed}moments of order $\leq k-2$ of each component  in   each element  $K\in \Tcal$:
\begin{align*}
\int_K \vfieldu\cdot \vfieldv,\quad \vfieldv\in P_{k -2}(K;\R^3) .
\end{align*}
\end{enumerate}
\begin{theorem}\label{thm:Honeuni}
The degrees of freedom \ref{enu:Honea}--\ref{enu:Honed} are unisolvent for  $P_{k +2}(K;\R^3)$.
\end{theorem}
\begin{proof}A direct computation shows that the number of the degrees of freedom is equal to the dimension of $P_{k+2}(K;\R^3)$, namely
\begin{align*}
120+18(k -3)+&6(k ^2+k -6) +\frac{(k -1)k (k +1)}{2}=\frac{(k +3)(k +4)(k +5)}{2}.
\end{align*} 

It suffices to prove if \ref{enu:Honea}--\ref{enu:Honed} vanish  for $\vfieldu\in P_{m+1}(K;\R^3)$ then $\vfieldu=0$.  
For each $f\in\Fcal(K)$, due to $\vfieldu \in P_{k +2}(f;\R^3)$, similar arguments as in Lemma \ref{lem:uniHone} plus \ref{enu:Honea}--\ref{enu:Honec} lead to $\vfieldu =0$ on $f$.  
The zero boundary condition $\vfieldu=0$ on $\partial K$ shows that there exists some $\vfieldw\in P_{k -2}(K;\R^3)$  such that
\begin{align*}
\vfieldu=\lambda_1\lambda_2\lambda_3\lambda_4 \vfieldw.
\end{align*} 
This and the degrees of freedom in \ref{enu:Honed} lead to $\vfieldu\equiv 0$ in $K$. This concludes the proof.
\end{proof}

The conforming finite element space $V_{k +2,h}\subset H^1(\Omega;\R^3)$ is defined as
\begin{align}
\label{eq:HoneS}
\begin{aligned}
V_{k +2,h}:=&\{ \vfieldv_h\in H^1(\Omega;\R^3)\ |\ \vfieldv_h|_K\in P_{k +2}(K;\R^3)\text{ for all }K\in\Tcal,\\
&\text{ all the degrees of freedom  \ref{enu:Honea}--\ref{enu:Honed} are single-valued}\}.
\end{aligned}
\end{align}

 Before establishing the conforming finite element $\ddiv\ddiv$ complexes with respect to  \eqref{eq:complex3Dintrodis}, the following theorem proves that the bubble function spaces on each element $K$ form  exact complexes. Recall the bubble function space $B_{k+1,\sym\ccurl}(K)$ of $H(\sym\ccurl,\Tbb)$ from \eqref{eq:bubsymcurl}. Let $B_{k +2,\ddev\nabla}(K)$ denote the bubble function space of the vectorial  $H^1$ space with vanishing degrees of freedom \ref{enu:Honea}--\ref{enu:Honec}. It is easy to check that
 \begin{align*}
 B_{k +2,\ddev\nabla}(K):=\lambda_1\lambda_2\lambda_3\lambda_4P_{k-2}(K;\R^3).
\end{align*}  Let  $B_{k , \ddiv\ddiv}(K)$ denote the bubble function space of $H(\ddiv\ddiv,\Sbb)$ with vanishing degrees of freedom \ref{enu:Hdivdiva}--\ref{enu:Hdivdive}. The  following lemma plus (b) of  Lemma \ref{lem:symcurlrestri}  imply the inclusion $ \ddev\nabla B_{k +2,\ddev\nabla}(K)\subset B_{k+1,\sym\ccurl}(K)$.
\begin{lemma}\label{lem:restrictone}
Suppose $\tau=\ddev\nabla v$. Then, on edge $e$ with the unit tangential vector  $\tvec=\nvec_1\times \nvec_2 $,
\begin{align}
\label{eq:edgegradone}
\nvec_i^T\sym\ccurl  \mfieldt \nvec_j=0,\quad i,j=1,2,\\
\label{eq:edgegradtwo}
\nvec_1^T\ccurl \mfieldt \nvec_2 -\partial_t(\tvec^T\mfieldt \tvec)=-\partial^2_{tt}(\vfieldv\cdot \tvec).
\end{align}
\end{lemma}
\begin{proof}
Given $e$, $\sym\ccurl\mfieldt=0$ results in \eqref{eq:edgegradone}. Note that
\begin{align*}
\nvec_1^T\ccurl \mfieldt\nvec_{2} -\partial_t(\tvec^T\mfieldt \tvec)=\frac{1}{3}\nvec_1^T\ccurl(\ddiv \vfieldv \mfieldI)\nvec_2-\partial_t(\tvec^T\nabla \vfieldv \tvec-\frac{1}{3}\ddiv\vfieldv).
\end{align*}
The cross product rule plus  $\tvec=\nvec_1\times \nvec_2$ lead to
\begin{align*}
\nvec_1^T\ccurl(\ddiv \vfieldv \mfieldI)\nvec_2&=\ccurl\big((\ddiv \vfieldv) \nvec_1\big) \cdot \nvec_2=\ddiv\big((\ddiv \vfieldv) \nvec_2\times \nvec_1\big)\\
&=-\ddiv((\ddiv\vfieldv) \tvec)=-\partial_t(\ddiv\vfieldv).
\end{align*}
The previous two identities conclude \eqref{eq:edgegradtwo}.
\end{proof}

\begin{theorem}\label{thm:elementbubbleseq}
For $k\geq 3$, it holds that
\begin{align*}\Scale[0.9]{
0\xrightarrow[]{ \subset} B_{k +2,\ddev\nabla}(K)\xrightarrow[]{ \ddev\nabla} B_{k +1,\sym\ccurl}(K) \xrightarrow[]{ \sym\ccurl} B_{k  ,\ddiv\ddiv}(K)\xrightarrow[]{ \ddiv\ddiv}P_{k-1}(K)/P_1(K)\rightarrow 0.}
\end{align*}
\end{theorem}
\begin{proof}
Lemma~\ref{lem:restrictone} plus (b) of  Lemma \ref{lem:symcurlrestri} show  the inclusion $\ddev\nabla B_{k +2,\ddev\nabla}(K)\subset B_{k +1,\sym\ccurl}(K)$ and the proof of Theorem \ref{thm:Hsymcurluni} shows the inclusion $\sym\ccurl B_{k +1,\sym\ccurl}(K)\subset B_{k  ,\ddiv\ddiv}(K)$.
 Suppose $\mfieldt\in \Lambda_{ k +1,0}(K)$ and $\sym\ccurl\mfieldt=0$.
The proof of Theorem~\ref{thm:Hsymcurluni} shows that there exists some  $u\in B_{k +2,\ddev\nabla}(K)$ such that $\mfieldt=\ddev\nabla u$.

The previous arguments show  $B_{k +1,\sym\ccurl}(K)\cap\ker(\sym\ccurl)=\ddev\nabla B_{k +2,\ddev\nabla}(K)$. This also means
\begin{align*}
\dim\sym\ccurl B_{k +1,\sym\ccurl}(K)&=\dim B_{k +1,\sym\ccurl}(K)-\dim B_{k +2,\ddev\nabla}(K)\\
&=\frac{k (k -1)(5k +14)}{6}+\frac{k (k -1)}{2}.
\end{align*}
It has been proved in \cite{ChenHuang20203D} that  $\ddiv\ddiv B_{k  ,\ddiv\ddiv}(K)=P_{k -2}(K)/P_1(K)$. This implies
\begin{align*}
\dim\big(B_{k  ,\ddiv\ddiv}(K)\cap\ker(\ddiv\ddiv)\big)&=\dim B_{k  ,\ddiv\ddiv}(K)-\dim P_{k -2}(K)/P_1(K)\\
&=\frac{k (k -1)(5k +14)}{6}+\frac{k (k -1)}{2}.
\end{align*}
This concludes that the complexes are exact.
\end{proof}

The following theorem states the finite element $\ddiv\ddiv$ complexes with respect to \eqref{eq:complex3D}.
\begin{theorem}\label{thm:disdivdivcomplex}
For $k\geq 3$, it holds that
\begin{align*}
RT\xrightarrow[]{ \subset}V_{k +2,h} \xrightarrow[]{ \ddev\nabla}\Lambda_{ k +1,h}  \xrightarrow[]{ \sym\ccurl}\Sigma_{k ,h}  \xrightarrow[]{ \ddiv\ddiv} P_{k-2}(\T)\rightarrow 0.
\end{align*}
\end{theorem}
\begin{proof}Similar arguments as in the proof of Theorem \ref{thm:elementbubbleseq} show the inclusions $\ddev\nabla V_{k +2,h}\subset\Lambda_{ k +1,h}$ and  $\sym\ccurl\Lambda_{ k +1,h}\subset\Sigma_{k , h}$.
 Suppose $\mfieldt\in \Lambda_{k +1,h } $ and $\sym\ccurl\mfieldt=0$.  
  The continuous complex \eqref{eq:complex3D} shows that there exists $\vfieldu\in H^1(\Omega;\R^3)/ RT$ such that
\begin{align*}
\mfieldt=\ddev\nabla \vfieldu.
\end{align*}
As shown in the proof of \cite[Lemma 3.2]{ChenHuang20203D}, $\vfieldu\in   P_{k +2}(K;\R^3)$ for any $K\in\Tcal$. Since $\vfieldu$ is a discrete function, $\vfieldu$ is continuous at vertices, and on edges and faces. It only remains to prove the extra continuity of the first and second order derivatives of $\vfieldu$ at vertices.  It is   shown in \cite[Lemma 3.2]{ChenHuang20203D}
\begin{align*}
{\rm mspn\ }(\nabla \ddiv \vfieldu)=3\ccurl\mfieldt.
\end{align*}
This leads to the continuity of $\nabla (\ddiv \vfieldu)$ at vertices and hence leads to the continuity of $\nabla^2\vfieldu$  at vertices. Given $f\in\Fcal$, the continuity of $\nabla_f(\vfieldu\cdot\nvec)$ and $\epsilon_f(\vfieldu\times\nvec)$ imply the continuity of $\nabla \vfieldu$ at vertices. The previous arguments lead to $\vfieldu\in V_{k+2,h}$ and consequently  $\Lambda_{ k +1,h}\cap\ker(\sym\ccurl)=\ddev\nabla V_{k +2,h}$. This also means
\begin{align*}
\dim\sym\ccurl &\Lambda_{ k +1,h}=\dim \Lambda_{ k +1,h}-\dim V_{k +2,h}/RT\\
&=2\#\Vcal+ (3k +1)\#\Ecal+(k ^2- k -3)\#\Fcal+\frac{ 5k ^3+12k ^2-17k }{6}\#\Tcal+4. 
\end{align*}
It  has been proved in \cite{ChenHuang20203D} that $\ddiv\ddiv\Sigma_{k  ,h} =P_{k -2} (\T)$. This shows
\begin{align*}
\dim&\big(\Sigma_{k ,h} \cap\ker(\ddiv\ddiv)\big)=\dim\Sigma_{k ,h} -\dim Q_{k -2,h} \\
=& 6\#\Vcal+(3k -3)\#\Ecal+(k ^2-k +1)\#\Fcal +\frac{ 5k ^3+12k ^2-17k -24}{6}\#\Tcal.
\end{align*}
The Euler's formula is $ \# V- \#E+ \#F- \#K-1=0$ for simple domains.
Since $\sym\ccurl\Lambda_{ k +1,h}\subseteq\Sigma_{k  ,h}\cap\ker(\ddiv\ddiv)$, this concludes that the complexes are exact.
\end{proof}

\section{The Dual formulation of the  linearized Einstein-Bianchi system}\label{sec:EB}This sections considers the discretization of \eqref{eq:strongformEB} in a   weak formulation, which is dual to the formulation in \cite{huliang2020}. The error analysis is provided by following similar arguments  as in \cite{huliang2020}. 
\subsection{Weak formulation}
The dual formulation of the linearized Einstein-Bianchi system introduced in \cite{huliang2020} reads: Find 
\begin{align}
\begin{aligned}
\sigma&\in C^1\big([0,T],L^2(\Omega)\big),\\
\mfieldE&\in C^0\big([0,T], H(\ddiv\ddiv,\Omega;\Sbb)\big)\cap C^1\big([0,T],L^2(\Omega;\Sbb)\big),\\
\mfieldB&\in C^0\big([0,T],H(\sym\ccurl,\Tbb)\big)\cap C^1\big([0,T],L^2(\Omega,\Tbb)\big),
\end{aligned}
\end{align}
such that
\begin{align}\label{eq:EinBian}
\begin{cases}
(\dot{\sigma},q)=(\ddiv\ddiv\mfieldE,q), &\text{ for any }q\in L^2(\Omega),\\
(\dot{\mfieldE},\mfieldxi)=-(\sigma,\ddiv\ddiv\mfieldxi)-(\sym\ccurl\mfieldB,\mfieldxi),&\text{ for any }\mfieldxi\in H(\ddiv\ddiv,\Omega;\Sbb),\\
(\dot{\mfieldB},\mfieldzeta)=(\mfieldE,\sym\ccurl\mfieldzeta),&\text{ for any }\mfieldzeta\in H(\sym\ccurl,\Omega;\Tbb).
\end{cases}
\end{align}

For $k\geq 3$, the semidiscretization of \eqref{eq:EinBian} finds
\begin{align*}
\sigma_h\in C^1\big([0,T ],P_{k-2}(\T)\big),\;\mfieldE_h\in  C^0\big([0,T ],\Sigma_{k,h}\big) \text{ and } \mfieldB_h\in  C^0\big([0,T ],\Lambda_{k+1,h}\big)
\end{align*}
such that
\begin{align}\label{eq:discreteEinBian}
\begin{cases}
(\dot{\sigma_h},q)=(\ddiv\ddiv\mfieldE_h,q), &\text{ for any }q\in P_{k-2}(\T),\\
(\dot{\mfieldE_h},\mfieldxi)=-(\sigma_h,\ddiv\ddiv\mfieldxi)-(\sym\ccurl\mfieldB_h,\mfieldxi),&\text{ for any }\mfieldxi\in \Sigma_{k,h},\\
(\dot{\mfieldB_h},\mfieldzeta)=(\mfieldE_h,\sym\ccurl\mfieldzeta),&\text{ for any }\mfieldzeta\in \Lambda_{k+1,h},
\end{cases}
\end{align}
for all $t\in (0,T]$, with given initial data.
\begin{theorem}
There exists a unique solution to \eqref{eq:discreteEinBian}.
\end{theorem}
\begin{proof}
The proof follows the same argument as in \cite[Theorem 6.1]{huliang2020}.
\end{proof}

\bigskip
Below investigates the convergence of the discrete solutions. Let $\Vspace:=L^2(\Omega)\times H(\ddiv\ddiv,\Omega;\Sbb)\times H(\sym\ccurl,\Omega;\Tbb)$ with the norm
\begin{align*}
\Lnorm{(q,\mfieldxi,\mfieldzeta)}_{\Vspace}:=\Lnorm{q}_{L^2(\Omega)}+\Lnorm{\mfieldxi}_{H(\ddiv\ddiv,\Omega)}+\Lnorm{\mfieldzeta}_{H(\sym\ccurl,\Omega)} 
\end{align*}
for  $(q,\mfieldxi,\mfieldzeta)\in\Vspace$. Let $\Vspaceh:=P_{k-2}(\T)\times\Sigma_{k,h}\times\Lambda_{k+1,h}$.
Define the bilinear form $\Acal: \Vspace\times \Vspace\rightarrow\R$ by
\begin{align*}
\Acal(\sigma,\mfieldE,\mfieldB;q,\mfieldxi,\mfieldzeta)=&(\sigma,q)+(\mfieldE,\mfieldxi)+(\mfieldB,\mfieldzeta)-(\ddiv\ddiv\mfieldE,q)+(\sigma,\ddiv\ddiv\mfieldxi)\\
&+(\sym\ccurl\mfieldB,\mfieldxi)-(E,\sym\ccurl\mfieldzeta).
\end{align*}
The following theorem shows the inf-sup condition of $\Acal$ in $\Vspace_h\times\Vspace_h$.

\begin{theorem}
The bilinear form $\Acal$ satisfies the inf-sup condition
\begin{align*}
\inf_{0\neq (\sigma,\mfieldE,\mfieldB)\in \Vspaceh}\sup_{0\neq (q,\mfieldxi,\mfieldzeta)\in \Vspaceh}\frac{\Acal(\sigma,\mfieldE,\mfieldB;q,\mfieldxi,\mfieldzeta)}{\Lnorm{(\sigma,\mfieldE,\mfieldB)}_{\Vspace}\Lnorm{(q,\mfieldxi,\mfieldzeta)}_{\Vspace}}=\beta>0
\end{align*}
with constant $\beta$ independent of $h$.
\end{theorem}
\begin{proof}
For any $(\sigma,\mfieldE,\mfieldB)\in \Vspaceh$, the finite element complexes   in Theorem \ref{thm:disdivdivcomplex} show $(q,\mfieldxi,\mfieldzeta)=(\sigma-\ddiv\ddiv\mfieldE,\mfieldE+\sym\ccurl\mfieldB,\mfieldB)\in \Vspaceh$. Thus, there exists   some positive constant $C$ such that
\begin{align*}
&\Acal(\sigma,\mfieldE,\mfieldB;q,\mfieldxi,\mfieldzeta)=(\sigma,\sigma)+(\mfieldE,\mfieldE)+(\mfieldB,\mfieldB)-(\sigma,\ddiv\ddiv\mfieldE)+(\mfieldE,\sym\ccurl\mfieldB)\\
&+(\ddiv\ddiv\mfieldE,\ddiv\ddiv\mfieldE)+(\sym\ccurl\mfieldB,\sym\ccurl\mfieldB)\\
&\geq\frac{1}{2}\big(\Lnorm{\sigma}_{L^2(\Omega)}^2+\Lnorm{\mfieldE}_{L^2(\Omega)}^2+\Lnorm{\mfieldB}_{L^2(\Omega)}^2+\Lnorm{\ddiv\ddiv\mfieldE}_{L^2(\Omega)}^2+\Lnorm{\sym\ccurl\mfieldB}_{L^2(\Omega)}^2\big)\\
&\geq C\Lnorm{(\sigma,\mfieldE,\mfieldB)}^2_{\Vspace}.
\end{align*}
Since $\Lnorm{(q,\mfieldxi,\mfieldzeta)}_{\Vspace}\leq C\Lnorm{(\sigma,\mfieldE,\mfieldB)}_{\Vspace}$, this concludes the inf-sup condition with some constant $\beta>0$ independent of $h$.
\end{proof}

For any $(\sigma,\mfieldE,\mfieldB)\in \Vspace$, define the projection $\Pi_h(\sigma,\mfieldE,\mfieldB)\in\Vspaceh$ such that
\begin{align}\label{eq:probilinear}
\Acal(\Pi_h\sigma,\Pi_h\mfieldE,\Pi_h\mfieldB;q,\mfieldxi,\mfieldzeta)=\Acal(\sigma,\mfieldE,\mfieldB;q,\mfieldxi,\mfieldzeta)\text{ for any }(q,\mfieldxi,\mfieldzeta)\in\Vspaceh.
\end{align} with the full equivalent formulation  
\begin{align}
\label{eq:fullprobilinear}
\begin{cases}
(\Pi_h\sigma,q)\hspace{-0.5mm} -\hspace{-0.5mm} (\ddiv\ddiv\Pi_h\mfieldE ,q)=(\sigma,q) \hspace{-0.5mm} -\hspace{-0.5mm} (\ddiv\ddiv\mfieldE ,q), &\hspace{-0.5mm} \text{ for any }q\in P_{k-2}(\T),\\
(\Pi_h\mfieldE ,\mfieldxi)+(\Pi_h\sigma ,\ddiv\ddiv\mfieldxi)+(\sym\ccurl\Pi_h\mfieldB ,\mfieldxi)&\\
\quad\quad=( \mfieldE  ,\mfieldxi)+(\sigma ,\ddiv\ddiv\mfieldxi)+(\sym\ccurl\mfieldB ,\mfieldxi),&\hspace{-0.5mm} \text{ for any }\mfieldxi\in \Sigma_{k,h},\\
(\Pi_h\mfieldB ,\mfieldzeta)\hspace{-0.5mm} -\hspace{-0.5mm} (\mfieldE_h,\sym\ccurl\mfieldzeta)\hspace{-0.5mm} =\hspace{-0.5mm} ( \mfieldB,\mfieldzeta)\hspace{-0.5mm} -\hspace{-0.5mm} (\mfieldE ,\sym\ccurl\mfieldzeta), &\hspace{-0.5mm} \text{ for any }\mfieldzeta\in \Lambda_{k+1,h}.
\end{cases}
\end{align}
The following error estimate holds from the Babu\u{s}ka theory \cite{bbf}
\begin{align}
\label{eq:bestapr}
\Lnorm{(\sigma,\mfieldE,\mfieldB)-\Pi_h(\sigma,\mfieldE,\mfieldB)}_{\Vspace}\lesssim \inf_{(q,\mfieldxi,\mfieldzeta)\in\Vspaceh}\Lnorm{(\sigma,\mfieldE,\mfieldB)- (q,\mfieldxi,\mfieldzeta)}_{\Vspace}
\end{align}

Let $P_V: L^2(\Omega)\rightarrow P_{k-2}(\T)$ denote the $L^2$ projection onto $P_{k-2}(\T)$ . For any $0\leq s\leq k-1$, the estimate holds
\begin{align}
\label{eq:estLtwo}
\Lnorm{q-P_Vq}_{L^2(\Omega)}\lesssim h^{s}\Lnorm{q}_{H^{s}(\Omega)}\text{ for any }q\in H^{s}(\Omega).
\end{align} Let $P_\Sigma: H(\ddiv\ddiv,\Omega;\Sbb)\cap H^2(\Omega;\Sbb)\rightarrow \Sigma_{k,h} $ with $k\geq 3$ denote the interpolation indicated by  the degrees of freedom   \ref{enu:Hdivdiva}--\ref{enu:Hdivdivf} (the values at the vertices are obtained by averaging). An alternative interpolation can be found in \cite{ChenHuang20203D} with the commuting property. The following error estimate holds, for any $0\leq s\leq k-1$, that
\begin{align}
\label{eq:estdivdiv}
\Lnorm{\mfieldxi-P_\Sigma\mfieldxi}_{H(\ddiv\ddiv,\Omega)}\lesssim h^{s}\Lnorm{\mfieldxi}_{H^{s+2}(\Omega)}\text{ for any }\mfieldxi\in H^{s+2}(\Omega).
\end{align}
 Let $P_\Lambda:H(\sym\ccurl,\Omega;\Tbb)\cap H^3(\Omega;\Tbb)\rightarrow \Lambda_{k+1,h}$ with $k\geq 3$ denote the interpolation by  the degrees of freedom   \ref{enu:HScurla}--\ref{enu:HScurlf} (the values at the vertices are obtained by averaging). The following error estimate holds, for any $2\leq s\leq k+1$, that
 \begin{align}
\label{eq:estsymcurl}
\Lnorm{\mfieldzeta-P_\Lambda\mfieldzeta}_{H(\sym\ccurl,\Omega)}\lesssim h^{s}\Lnorm{\mfieldzeta}_{H^{s+1}(\Omega)}\text{ for any }\mfieldxi\in H^{s+1}(\Omega;\Tbb).
\end{align}
Suppose $(\sigma,\mfieldE,\mfieldB)\in H^{k-1}(\Omega)\times H^{k+1}(\Omega;\Sbb)\times H^{k}(\Omega)$ for $k\geq 3$.
The combination of \eqref{eq:estLtwo}--\eqref{eq:estsymcurl} with \eqref{eq:bestapr} leads to
\begin{align}
\label{eq:project}
\begin{aligned}
\Lnorm{(\sigma,\mfieldE,\mfieldB)\hspace{-0.5mm}-\hspace{-0.5mm}\Pi_h(\sigma,\mfieldE,\mfieldB)}_{\Vspace}
\lesssim& h^{k-1}\big(\Lnorm{\sigma}_{H^{k-1}(\Omega)}\hspace{-0.5mm}+\Lnorm{\mfieldE}_{H^{k+1}(\Omega)}+\Lnorm{\mfieldB}_{H^{k}(\Omega)}\big).\hspace{-0.5mm}
\end{aligned}
\end{align}
\subsection{The solution of the fully discrete system and error estimates}This subsection uses some notation in \cite{huliang2020}.  Suppose that $T=N\Delta t$ with a positive integer $N$. Let $p^j$ denote the function $p(t_j)$ with $t_j=j\Delta t$ for $j=0,1,\cdots,N$.  Define
\begin{align*}
\partial_t p^{j+\frac{1}{2}}=\frac{p^{j+1}-p^j}{\Delta t},\;\hat{p}^{j+\frac{1}{2}}=\frac{p^{j+1}+p^j}{2}.
\end{align*}

The time variable will be discretized by the Crank-Nicolson scheme. Denote $(\sigma_h^j,\mfieldE_h^j,\mfieldB_h^j)\in\Vspaceh$ the approximation of solution $(\sigma,\mfieldE,\mfieldB)$ of \eqref{eq:EinBian} at $t_j$. Given the initial data $(\sigma_h^0,\mfieldE_h^0,\mfieldB_h^0)\in\Vspaceh$, for $0\leq j\leq N-1$, the approximation $(\sigma_h^{j+1},\mfieldE_h^{j+1},\mfieldB_h^{j+1})$ at $t_{j+1}$ is defined by
\begin{align}\label{eq:fulldis}
\begin{cases}
(\partial_t\sigma_h^{j+\frac{1}{2}},q)=(\ddiv\ddiv\hat{\mfieldE}_h^{j+\frac{1}{2}},q) ,&\text{ for any }q\in P_{k-2}(\T),\\
(\partial_t\mfieldE_h^{j+\frac{1}{2}} ,\mfieldxi)=-(\hat{\sigma}_h^{j+\frac{1}{2}} ,\ddiv\ddiv\mfieldxi)-(\sym\ccurl\hat{\mfieldB}_h^{j+\frac{1}{2}},\mfieldxi),\hspace{-3.5mm}&\text{ for any }\mfieldxi\in \Sigma_{k,h},\\
(\partial_t\mfieldB_h^{j+\frac{1}{2}} ,\mfieldzeta)=(\hat{\mfieldE}_h^{j+\frac{1}{2}},\sym\ccurl\mfieldzeta),&\text{ for any }\mfieldzeta\in \Lambda_{k+1,h}.
\end{cases}
\end{align}
This can be written as
\begin{align*}
\begin{cases}
(\sigma_h^{j+1},q)-\frac{\Delta t}{2}(\ddiv\ddiv\mfieldE_h^{j+1},q)=(\sigma_h^{j},q)+\frac{\Delta t}{2}(\ddiv\ddiv\mfieldE_h^{j},q),&\text{ for any }q\in P_{k-2}(\T),\\
( \mfieldE_h^{j+1} ,\mfieldxi)+\frac{\Delta t}{2}(\sigma_h^{j+1} ,\ddiv\ddiv\mfieldxi)+\frac{\Delta t}{2}(\sym\ccurl \mfieldB_h^{j+1},\mfieldxi)&\\
\quad\quad=( \mfieldE_h^{j } ,\mfieldxi)-\frac{\Delta t}{2}(\sigma_h^{j} ,\ddiv\ddiv\mfieldxi)-\frac{\Delta t}{2}(\sym\ccurl \mfieldB_h^{j},\mfieldxi),&\text{ for any }\mfieldxi\in \Sigma_{k,h},\\
( \mfieldB_h^{j+1} ,\mfieldzeta)-\frac{\Delta t}{2}( \mfieldE_h^{j+1},\sym\ccurl\mfieldzeta)=( \mfieldB_h^{j } ,\mfieldzeta)+\frac{\Delta t}{2}( \mfieldE_h^{j },\sym\ccurl\mfieldzeta),&\text{ for any }\mfieldzeta\in \Lambda_{k+1,h}.
\end{cases}
\end{align*}
The system is nonsingular as in \cite{huliang2020}.
The error estimates mimic the proof of Theorem 6.3 of \cite{huliang2020} and  are stated in the following theorem.
\begin{theorem}Suppose $k\geq 3$. 
Let $(\sigma,\mfieldE,\mfieldB)$ solve \eqref{eq:EinBian} and let $(\sigma_h^j,\mfieldE_h^j,\mfieldB_h^j)$ solve \eqref{eq:fulldis}, let the initial data $(\sigma_h,\mfieldE_h^0,\mfieldB_h^0)=\Pi_h(\sigma(0),\mfieldE(0),\mfieldB(0))$. Assume
\begin{align*}
\sigma\in W^{1,1}\big([0,T], H^{k-1}(\Omega) \big)\cap W^{3,1}([0,T],L^2(\Omega))\cap L^\infty([0,T],H^{k-1}(\Omega)),\\
\mfieldE\in W^{1,1}\big([0,T], H^{k+1}(\Omega;\Sbb)\big)\cap W^{3,1}([0,T],L^2(\Omega;\Sbb))\cap L^\infty([0,T],H^{k+1}(\Omega;\Sbb)),\\
\mfieldB\in W^{1,1}\big([0,T], H^{k}(\Omega;\Tbb) \big)\cap W^{3,1}([0,T],L^2(\Omega;\Tbb))\cap L^\infty([0,T],H^{k}(\Omega;\Tbb)).
\end{align*}
It holds that, for $1\leq j\leq N$
\begin{align*}
&\Lnorm{\sigma^j-\sigma_h^j}_{L^2(\Omega)}+\Lnorm{\mfieldE^j-\mfieldE^j_h}_{L^2(\Omega)}+\Lnorm{\mfieldB^j-\mfieldB^j_h}_{L^2(\Omega)}\\
\lesssim &(h^{k-1}+\Delta t^2)\big(\Lnorm{\sigma}_{W^{1,1}(H^{k-1})\cap W^{3,1}(L^2)\cap L^\infty(H^{k-1})}+\Lnorm{\mfieldE}_{W^{1,1}(H^{k+1})\cap W^{3,1}(L^2)\cap L^\infty(H^{k+1})}\\
&+\Lnorm{\mfieldB}_{W^{1,1}(H^{k })\cap W^{3,1}(L^2)\cap L^\infty(H^{k })}\big)
\end{align*}
\end{theorem}
\begin{proof}
There exists the following decomposition of the errors
\begin{align*}
\delta^j_{\sigma}:&=\sigma_h^j-\sigma^j=(\sigma^j_h-\Pi_h\sigma^j)+(\Pi_h\sigma^j-\sigma^j):=\theta^j_{\sigma}+p^j_{\sigma},\\
\delta^j_{\mfieldE}:&=\mfieldE_h^j-\mfieldE^j=(\mfieldE^j_h-\Pi_h\mfieldE^j)+(\Pi_h\mfieldE^j-\mfieldE^j):=\theta^j_{\mfieldE}+p^j_{\mfieldE},\\
\delta^j_{\mfieldB}:&=\mfieldB_h^j-\mfieldB^j=(\mfieldB^j_h-\Pi_h\mfieldB^j)+(\Pi_h\mfieldB^j-\mfieldB^j):=\theta^j_{\mfieldB}+p^j_{\mfieldB}.
\end{align*}
The second terms in the above have been estimated in \eqref{eq:project}. It remains to analyze  the errors $(\theta^j_{\sigma},\theta^j_{\mfieldE},\theta^j_{\mfieldB})$.

The  choices of  $t=t_j$ and $t=t_{j+1}$ in \eqref{eq:EinBian} lead  to
\begin{align}
\begin{cases}
(\hat{\dot{\sigma}}^{j+\frac{1}{2}},q)=(\ddiv\ddiv\hat{\mfieldE}^{j+\frac{1}{2}},q),&\text{ for any }\tau\in P_{k-2}(\T)\\
(\hat{\dot{\mfieldE}}^{j+\frac{1}{2}},\mfieldxi)=-(\hat{\sigma}^{j+\frac{1}{2}},\ddiv\ddiv\mfieldxi)-(\sym\ccurl\hat{\mfieldB}^{j+\frac{1}{2}},\mfieldxi),&\text{ for any }\mfieldxi\in \Sigma_{k,h},\\
(\hat{\dot{\mfieldB}}^{j+\frac{1}{2}},\mfieldzeta)=(\hat{ \mfieldE }^{j+\frac{1}{2}},\sym\ccurl\mfieldzeta),&\text{ for any }\mfieldzeta\in\Lambda_{k+1,h}.
\end{cases}
\end{align}
Substracting \eqref{eq:fulldis} from the above equations shows
\begin{align}\label{eq:erreq}
\begin{cases}
(\partial_t \delta_{\sigma}^{j+\frac{1}{2}},q)+(\partial_t\sigma^{j+\frac{1}{2}}-\hat{\dot{\sigma}}^{j+\frac{1}{2}},q)=(\ddiv\ddiv\hat{\delta}_{\mfieldE}^{j+\frac{1}{2}},q) ,&\text{ for any }q\in P_{k-2}(\T),\\
(\partial_t\delta_{\mfieldE}^{j+\frac{1}{2}} ,\mfieldxi)+(\partial_t\mfieldE^{j+\frac{1}{2}}-\hat{\dot{\mfieldE}}^{j+\frac{1}{2}},\mfieldxi)&\\
\quad=-(\hat{\delta}_{\sigma}^{j+\frac{1}{2}} ,\ddiv\ddiv\mfieldxi)-(\sym\ccurl\hat{\delta}_{\mfieldB}^{j+\frac{1}{2}},\mfieldxi),&\text{ for any }\mfieldxi\in \Sigma_{k,h},\\
(\partial_t\delta_{\mfieldB}^{j+\frac{1}{2}} ,\mfieldzeta)+(\partial_t\mfieldB^{j+\frac{1}{2}}-\hat{\dot{\mfieldB}}^{j+\frac{1}{2}},\mfieldzeta)=(\hat{\delta}_{\mfieldE}^{j+\frac{1}{2}},\sym\ccurl\mfieldzeta),&\text{ for any }\mfieldzeta\in \Lambda_{k+1,h}.
\end{cases}
\end{align}
It follows from \eqref{eq:fullprobilinear} that
\begin{align*}
\begin{cases}
(\hat{p}_{\sigma}^{j+\frac{1}{2}},q)=(\ddiv\ddiv\hat{p}_{\mfieldE}^{j+\frac{1}{2}},q),&\text{ for any }q\in P_{k-2}(\T),\\
(\hat{p}_{\mfieldE}^{j+\frac{1}{2}},\mfieldxi)=-(\hat{p}_{\sigma}^{j+\frac{1}{2}},\ddiv\ddiv\mfieldxi)-(\sym\ccurl\hat{p}_{\mfieldB}^{j+\frac{1}{2}},\mfieldxi),&\text{ for any }\mfieldxi\in \Sigma_{k,h},\\
(\hat{p}_{\mfieldB}^{j+\frac{1}{2}},\mfieldzeta)=(\hat{p}_{\mfieldE}^{j+\frac{1}{2}},\sym\ccurl\mfieldzeta),&\text{ for any }\mfieldzeta\in \Lambda_{k+1,h}.
\end{cases}
\end{align*}

The choices of  $q=\hat{\theta}_{\sigma}^{j+\frac{1}{2}}$, $\mfieldxi=\hat{\theta}_{\mfieldE}^{j+\frac{1}{2}}$, $\mfieldzeta=\hat{\theta}_{\mfieldB}^{j+\frac{1}{2}}$ in \eqref{eq:erreq} and the above equations  plus an elementary computation lead to
\begin{align*}
\begin{aligned}
&\big(\Lnorm{\theta_{\sigma}^{j+1}}^2_{L^2(\Omega)}+\Lnorm{\theta_{\mfieldE}^{j+1}}^2_{L^2(\Omega)}+\Lnorm{\theta_{\mfieldB}^{j+1}}^2_{L^2(\Omega)}\big)-\big(\Lnorm{\theta_{\sigma}^{j }}^2_{L^2(\Omega)}+\Lnorm{\theta_{\mfieldE}^{j }}^2_{L^2(\Omega)}+\Lnorm{\theta_{\mfieldB}^{j }}^2_{L^2(\Omega)}\big)\\
=&2\Delta t\big(-\partial_t p_{\sigma}^{j+\frac{1}{2}}-(\partial_t\sigma^{j+\frac{1}{2}}-\hat{\dot{\sigma}}^{j+\frac{1}{2}})+\hat{p}_{\sigma}^{j+\frac{1}{2}},\hat{\theta}_{\sigma}^{j+\frac{1}{2}}\big)\\
&+2\Delta t\big(-\partial_tp_{\mfieldE}^{j+\frac{1}{2}}-(\partial_t\mfieldE^{j+\frac{1}{2}}-\hat{\dot{\mfieldE}}^{j+\frac{1}{2}})+\hat{p}_{\mfieldE}^{j+\frac{1}{2}},\hat{\theta}_{\mfieldE}^{j+\frac{1}{2}}\big)\\
&+2\Delta t\big(-\partial_tp_{\mfieldB}^{j+\frac{1}{2}}-(\partial_t\mfieldB^{j+\frac{1}{2}}-\hat{\dot{\mfieldB}}^{j+\frac{1}{2}})+\hat{p}_{\mfieldB}^{j+\frac{1}{2}},\hat{\theta}_{\mfieldB}^{j+\frac{1}{2}}\big).
\end{aligned}
\end{align*}
An application of the Cauchy-Schwarz inequality proves
\begin{align}\label{eq:timestep}
\begin{aligned}
&\big(\Lnorm{\theta_{\sigma}^{j+1}}^2_{L^2(\Omega)}\hspace{-1mm}+\hspace{-0.5mm}\Lnorm{\theta_{\mfieldE}^{j+1}}^2_{L^2(\Omega)}\hspace{-1mm}+\hspace{-0.5mm}\Lnorm{\theta_{\mfieldB}^{j+1}}^2_{L^2(\Omega)}\big)\hspace{-0.5mm}^{\frac{1}{2}}\hspace{-0.75mm}-\hspace{-0.5mm}\big(\Lnorm{\theta_{\sigma}^{j }}^2_{L^2(\Omega)}\hspace{-1mm}+\hspace{-0.5mm}\Lnorm{\theta_{\mfieldE}^{j }}^2_{L^2(\Omega)}\hspace{-1mm}+\hspace{-0.5mm}\Lnorm{\theta_{\mfieldB}^{j }}^2_{L^2(\Omega)}\big)\hspace{-0.5mm}^{\frac{1}{2}}\hspace{-0.5mm}\\
&\lesssim  \Delta t\big(\Lnorm{\partial_t p_{\sigma}^{j+\frac{1}{2}}}_{L^2(\Omega)}+\Lnorm{\partial_t\sigma^{j+\frac{1}{2}}-\hat{\dot{\sigma}}^{j+\frac{1}{2}}}_{L^2(\Omega)}+\Lnorm{\hat{p}_{\sigma}^{j+\frac{1}{2}}}_{L^2(\Omega)}\\
&\quad+\Lnorm{\partial_t p_{\mfieldE}^{j+\frac{1}{2}}}_{L^2(\Omega)}+\Lnorm{\partial_t\mfieldE^{j+\frac{1}{2}}-\hat{\dot{\mfieldE}}^{j+\frac{1}{2}}}_{L^2(\Omega)}+\Lnorm{\hat{p}_{\mfieldE}^{j+\frac{1}{2}}}_{L^2(\Omega)}\\
&\quad+\Lnorm{\partial_t p_{\mfieldB}^{j+\frac{1}{2}}}_{L^2(\Omega)}+\Lnorm{\partial_t\mfieldB^{j+\frac{1}{2}}-\hat{\dot{\mfieldB}}^{j+\frac{1}{2}}}_{L^2(\Omega)}+\Lnorm{\hat{p}_{\mfieldB}^{j+\frac{1}{2}}}_{L^2(\Omega)}\big).
\end{aligned}
\end{align}
Given $g\in C^3[0,T]$, the Taylor expansion of $g$ reads
\begin{align*}
\begin{aligned}
\Delta t\Lnorm{\partial_t g^{j+\frac{1}{2}}}_{L^2(\Omega)}&=\Lnorm{\int_{t_j}^{t_{j+1}}\dot{g}\,dt}_{L^2(\Omega)}\leq \int_{t_j}^{t_{j+1}}\Lnorm{\dot{g}}_{L^2(\Omega)}\,dt,\\
\Delta t\Lnorm{\partial_tg^{j+\frac{1}{2}} -\hat{\dot{g}}^{j+\frac{1}{2}} }_{L^2(\Omega)}&=\frac{1}{2}\Lnorm{2g^{j+1}-2g^j-\Delta t\dot{g}^{j+1}-\Delta t\dot{g}^j}_{L^2(\Omega)}\\
&\lesssim\Delta t^2\int_{t_j}^{t_{j+1}}\Lnorm{\dddot{g}}_{L^2(\Omega)}\,dt.
\end{aligned}
\end{align*}
A summation of \eqref{eq:timestep}  over all the  time intervals  plus  the above estimates lead to
\begin{align*}
\begin{aligned}
&\big(\Lnorm{\theta_{\sigma}^{j+1}}^2_{L^2(\Omega)}\hspace{-1mm}+\hspace{-0.5mm}\Lnorm{\theta_{\mfieldE}^{j+1}}^2_{L^2(\Omega)}\hspace{-1mm}+\hspace{-0.5mm}\Lnorm{\theta_{\mfieldB}^{j+1}}^2_{L^2(\Omega)}\big)\hspace{-0.5mm}^{\frac{1}{2}}\hspace{-0.75mm}-\hspace{-0.5mm}\big(\Lnorm{\theta_{\sigma}^{0 }}^2_{L^2(\Omega)}\hspace{-1mm}+\hspace{-0.5mm}\Lnorm{\theta_{\mfieldE}^{0 }}^2_{L^2(\Omega)}\hspace{-1mm}+\hspace{-0.5mm}\Lnorm{\theta_{\mfieldB}^{0 }}^2_{L^2(\Omega)}\big)\hspace{-0.5mm}^{\frac{1}{2}}\hspace{-0.5mm}\\
& \lesssim \int_0^{t_{j+1}}\big(\Lnorm{\dot{p}_{\sigma}}_{L^2(\Omega)}+\Lnorm{\dot{p}_{\mfieldE}}_{L^2(\Omega)}+\Lnorm{\dot{p}_{\mfieldB}}_{L^2(\Omega)}\big)\,dt\\
&\quad+ \Delta t^2\int_0^{t_{j+1}}\big(\Lnorm{\dddot{\sigma}}_{L^2(\Omega)}+\Lnorm{\dddot{\mfieldE}}_{L^2(\Omega)}+\Lnorm{\dddot{\mfieldB}}_{L^2(\Omega)}\big)\,dt\\
&\quad+\Delta t\sum^{j+1}_{m=0}\big(\Lnorm{p^m_{\sigma}}_{L^2(\Omega)}+\Lnorm{p^m_{\mfieldE}}_{L^2(\Omega)}+\Lnorm{p^m_{\mfieldB}}_{L^2(\Omega)}\big).
\end{aligned}
\end{align*}
Since the initial data $(\sigma_h^0,\mfieldE_h^0,\mfieldB_h^0)=\Pi_h(\sigma(0),\mfieldE(0),\mfieldB(0))$, it implies that $(\theta_{\sigma}^0,\theta_{\mfieldE}^0,\theta_{\mfieldB}^0)$ vanishes. By the estimates of the projection errors  in \eqref{eq:project}, this shows that
\begin{align*}
&\big(\Lnorm{\theta_{\sigma}^{j+1}}^2_{L^2(\Omega)}+\Lnorm{\theta_{\mfieldE}^{j+1}}^2_{L^2(\Omega)}+\Lnorm{\theta_{\mfieldB}^{j+1}}^2_{L^2(\Omega)}\big)^{\frac{1}{2}}\\
\lesssim &   h^{k-1}\int_0^{t_{j+1}}\big(\Lnorm{\dot{\sigma}}_{H^{k-1}(\Omega)}+\Lnorm{\dot{\mfieldE}}_{H^{k+1}(\Omega)}+\Lnorm{\dot{\mfieldB}}_{H^{k}(\Omega)}\big)\,dt\\
&+\Delta t^2\int_0^{t_{j+1}}\big(\Lnorm{\dddot{\sigma}}_{L^2(\Omega)}+\Lnorm{\dddot{\mfieldE}}_{L^2(\Omega)}+\Lnorm{\dddot{\mfieldB}}_{L^2(\Omega)}\big)\,dt\\
&+j\Delta th^{k-1}(\Lnorm{\sigma}_{L^\infty(H^{k-1})}+\Lnorm{\mfieldE}_{L^\infty(H^{k+1})}+\Lnorm{\mfieldB}_{L^\infty(H^{k})}) .
\end{align*}
A combination of  this and the estimates of the projection errors in \eqref{eq:project} completes the proof.
\end{proof}
\bibliographystyle{siamplain}
\bibliography{lit_mit_doi}

\begin{thebibliography}{10}

\bibitem{arnold2006}
{\sc D.~N. Arnold, R.~S. Falk, and R.~Winther}, {\em Finite element exterior
  calculus, homological techniques, and applications}, Acta Numer., 15 (2006),
  pp.~1--155, \href{http://dx.doi.org/10.1017/S0962492906210018}
  {doi:10.1017/S0962492906210018}.

\bibitem{arnold2021complexes}
{\sc D.~N. Arnold and K.~Hu}, {\em Complexes from complexes}, 2021,
  \href{http://arxiv.org/abs/2005.12437v2} {arXiv:2005.12437v2}.

\bibitem{bbf}
{\sc D.~Boffi, F.~Brezzi, and M.~Fortin}, {\em Mixed finite element methods and
  applications}, Springer, Heidelberg, 2013.

\bibitem{chen2020discrete}
{\sc L.~Chen and X.~Huang}, {\em Discrete {H}essian complexes in three
  dimensions}, 2020, \href{http://arxiv.org/abs/2012.10914} {arXiv:2012.10914}.

\bibitem{ChenHuang20202D}
{\sc L.~Chen and X.~Huang}, {\em Finite elements for divdiv-conforming
  symmetric tensors},  (2020), \href{http://arxiv.org/abs/2005.01271v1}
  {arXiv:2005.01271v1}.

\bibitem{ChenHuang20203D}
{\sc L.~Chen and X.~Huang}, {\em Finite elements for divdiv-conforming
  symmetric tensors in three dimensions},  (2020),
  \href{http://arxiv.org/abs/2007.12399} {arXiv:2007.12399}.

\bibitem{CHK2018}
{\sc S.~H. Christiansen, J.~Hu, and K.~Hu}, {\em Nodal finite element de rham
  complexes}, Numer. Math., 139 (2018), pp.~411--446.

\bibitem{FH2019}
{\sc T.~F\"{u}hrer and N.~Heuer}, {\em Fully discrete {DPG} methods for the
  {K}irchhoff–{L}ove plate bending model}, Comput. Methods in Appl. Mech.
  Eng., 343 (2019), pp.~550--571,
  \href{http://dx.doi.org/10.1016/j.cma.2018.08.041}
  {doi:10.1016/j.cma.2018.08.041}.

\bibitem{FHN2019}
{\sc T.~F\"{u}hrer, N.~Heuer, and A.~H. Niemi}, {\em An ultraweak formulation
  of the {K}irchhoff-{L}ove plate bending model and {DPG} approximation}, Math.
  Comp., 88 (2019), pp.~1587--1619.

\bibitem{Hu2015trianghigh}
{\sc J.~Hu}, {\em Finite element approximations of symmetric tensors on
  simplicial grids in {$\mathbb{R}^n$}: {T}he higher order case}, J. Comput.
  Math., 33 (2015), pp.~283--296.

\bibitem{huliang2020}
{\sc J.~Hu and Y.~Liang}, {\em Conforming discrete gradgrad-complexes in three
  dimensions}, 2020, \href{http://arxiv.org/abs/2008.00497} {arXiv:2008.00497}.

\bibitem{HuMaZhang2020}
{\sc J.~Hu, R.~Ma, and M.~Zhang}, {\em A family of mixed finite elements for
  the biharmonic equations on triangular and tetrahedral grids},  (2020),
  \href{http://arxiv.org/abs/2010.02638} {arXiv:2010.02638}.

\bibitem{HuZhang2014a}
{\sc J.~Hu and S.~Zhang}, {\em A family of conforming mixed finite elements for
  linear elasticity on triangular grids}, ar{X}iv, 1406.7457 (2014).

\bibitem{HuZhang2015tre}
{\sc J.~Hu and S.~Zhang}, {\em A family of symmetric mixed finite elements for
  linear elasticity on tetrahedral grids}, Sci. China Math., 58 (2015),
  pp.~297--307.

\bibitem{PaulyZ2020}
{\sc D.~Pauly and W.~Zulehner}, {\em The divdiv-complex and applications to
  biharmonic equations}, Applicable Analysis, 99 (2020), pp.~1579--1630,
  \href{http://dx.doi.org/10.1080/00036811.2018.1542685}
  {doi:10.1080/00036811.2018.1542685}.

\bibitem{Quenne}
{\sc V.~Quenneville-B\'{e}lair}, {\em A new approach to finite element
  simulation of general relativity}, PhD thesis, University of Minnesota,
  Minneapolis, 2015.

\end{thebibliography}
\end{document}